\RequirePackage{fix-cm}
\documentclass[smallextended]{svjour3} 
\pdfoutput=1 
\PassOptionsToPackage{square,sort, comma,numbers, compress}{natbib}% onecolumn (second format)
\smartqed  % flush right qed marks, e.g. at end of proof
\usepackage{graphicx}
\usepackage{amsmath,amssymb,bbm,url}
\usepackage{pgf, tikz}
\usepackage{versions}
\usepackage[hidelinks]{hyperref}
\usetikzlibrary{shapes, arrows.meta, positioning, bending, quotes}
\usetikzlibrary{arrows, automata}

\spnewtheorem*{conj}{Conjecture}{\bf}{\it}

\newcommand{\reals}{\mathbb{R}}
\newcommand{\AT}{A}

\DeclareMathOperator*{\dom}{dom}
\DeclareMathOperator*{\inte}{int}
\DeclareMathOperator*{\epi}{epi}

\DeclareMathOperator*{\argmin}{arg\,min}

\newcommand{\norm}[1]{\left\|#1\right\|}
\newcommand{\inner}[2]{\left\langle #1,#2 \right\rangle}
\newtheorem{fact}{Fact}
\newcommand{\cA}{\mathcal{A}}
\newcommand{\cB}{\mathcal{B}}

\newcommand{\cO}{\mathcal{O}}

\newcommand{\cU}{\mathcal{U}}
\newcommand{\cV}{\mathcal{V}}

\newcommand{\RR}{\mathbb{R}}

\begin{document}
\title{
Mirror Duality in Convex Optimization
}
\authorrunning{Kim, Park, Ozdaglar, Diakonikolas, and Ryu}
\titlerunning{Mirror Duality in Convex Optimization} 
\author{
Jaeyeon Kim
\and 
Chanwoo Park 
\and
Asuman Ozdaglar
\and
\\
Jelena Diakonikolas
\and
Ernest K. Ryu
}

\institute{Jaeyeon Kim \at
              Seoul National University\\
              \email{kjy011102@snu.ac.kr}           %  \\
           \and
Chanwoo Park \at
              MIT\\
              \email{cpark97@mit.edu}           %  \\
           \and
Asuman Ozdaglar \at
              MIT\\
              \email{asuman@mit.edu}           %  \\
                \and
Jelena Diakonikolas \at
              University of Wisconsin--Madison\\
              \email{jdiakonikola@wisc.edu}           %  \\
                \and
Ernest K. Ryu \at
              Seoul National University\\
              \email{ernestryu@snu.ac.kr} 
}

\date{Received: date / Accepted: date}

\makeatletter

\renewcommand{\paragraph}{%
  \@startsection{paragraph}{4} {\z@}
  {1ex \@plus 1ex \@minus .2ex}
  {-1ex}%
  {\bfseries}%
}                                  
 
\makeatother
\maketitle
\begin{abstract}
While first-order optimization methods are usually designed to efficiently reduce the function value $f(x)$, there has been recent interest in methods efficiently reducing the magnitude of $\nabla f(x)$, and the findings show that the two types of methods exhibit a certain symmetry. In this work, we present mirror duality, a one-to-one correspondence between mirror-descent-type methods reducing function value and reducing gradient magnitude. Using mirror duality, we obtain the dual accelerated mirror descent (dual-AMD) method that efficiently reduces $\psi^*(\nabla f(x))$, where $\psi$ is a distance-generating function and $\psi^*$ quantifies the magnitude of $\nabla f(x)$. We then apply dual-AMD to efficiently reduce $\|\nabla f(\cdot) \|_q$ for $q\in [2,\infty)$ and to efficiently compute $\varepsilon$-approximate solutions of the optimal transport problem.

\end{abstract}
\section{Introduction}
Consider the problem
\begin{align*}
\begin{array}{ll}
\underset{x \in X}{\mbox{minimize}}&\quad f(x),
\end{array}
\end{align*} 
where $f\colon X\to\RR$ is a differentiable convex function and $X$ is a reflexive Banach space. One of the most fundamental facts in convex optimization is that every stationary point is also a global minimum, so the minimization problem is equivalent to
\begin{align*}
\begin{array}{ll}
\underset{x \in X}{\mbox{find}}&\quad 0=\nabla f(x).
\end{array}
\end{align*}

However, the problem of efficiently approximating function minima, i.e., making $f(x)-\inf_{x\in X}f(x)$ small, is quite different from the problem of efficiently computing approximate stationary points, i.e., making $\nabla f(x)$ small. Methods that exhibit optimal convergence rates under one of the criteria do not, in general, exhibit optimal convergence rates under the other. In particular, Nesterov's accelerated gradient method is an optimal method for the former but not the latter criterion.  Recently, methods like OGM-G \cite{kim2021optimizing} have been proposed as accelerated methods for reducing $\|\nabla f(x)\|_2$ when $X$ is the Euclidean space. However, the problem of finding approximate stationary points is still much less understood than the problem of finding approximate function minima in more general normed spaces. 

On the other hand, the framework of mirror descent has been a cornerstone to designing efficient methods for optimization problems with structures that are well-behaved in non-Euclidean geometries. However, compared to the large body of research on mirror-descent-type methods for efficiently reducing function value, which includes accelerated mirror descent, mirror-descent-type methods for reducing gradient magnitude have been much less explored.

In this work, we start from the perspective that reducing gradient magnitude \emph{is} reducing a function value in an appropriate mirror-descent setup. This observation, due to  \cite{doi:10.1137/19M130858X} and restated in Section~\ref{sec::disc_md}, more specifically is that exchanging the roles of the objective function and the conjugate of the distance-generating function (DGF) in (unaccelerated) mirror descent yields a method that reduces the gradient magnitude as measured by the DGF. 

We then build upon this insight and present mirror duality, a one-to-one correspondence between mirror-descent-type methods reducing function value and those reducing gradient magnitude. Using mirror duality, we obtain dual accelerated mirror descent \eqref{eqn::AMD-G} as the \emph{mirror dual} of accelerated mirror descent \eqref{eqn::AMD}. With mirror duality, the prior analysis showing that \ref{eqn::AMD} efficiently reduces function value implies that \ref{eqn::AMD-G} efficiently reduces $\psi^*(\nabla f(x))$, where $\psi$ is a distance-generating function and $\psi^*$ quantifies the magnitude of $\nabla f(x)$. Finally, we show how to apply \ref{eqn::AMD-G} to efficiently reduce $\|\nabla f(\cdot) \|_q$ for $q\in [2,\infty)$ and to efficiently compute approximate solutions of the optimal transport problem.

\newpage

\subsection{Preliminaries} 
\label{sec::convex_Banach}

We quickly review relevant basic concepts and set up the notation. 

\paragraph{Banach spaces.}
We follow the standard definitions of convex analysis in Banach spaces \cite{barbu2012}. Let $(X,\norm{\cdot})$ be a Banach space, $(X^{*},\norm{\cdot}_{*})$ be its (continuous) dual space, and $\inner{\cdot}{\cdot}\colon X^*\times X \rightarrow \mathbb{R}$ be the canonical pairing between $X^*$ and $X$. In finite-dimensional setups, $X=X^*=\mathbb{R}^n$, $\norm{\cdot}$ and $\norm{\cdot}_*$ are dual norms, and $\langle u,x\rangle =u^\intercal x$ is the standard Euclidean inner product. We assume $X$ is reflexive: $X^{**}=X$ and the dual norm of $\norm{\cdot}_{*}$ is $\norm{\cdot}$, which always holds when $X$ is finite-dimensional. We denote the Euclidean norm by $\norm{\cdot}_2$.

\paragraph{Basic convex analysis.}
A set $C \subset X$ is convex if $\lambda x + (1-\lambda)y \in C$ for all $x,y \in C , \, \lambda \in (0,1)$. The interior of $C\subseteq X$ is $\inte C = \{x\in C \,|\,\exists\, \epsilon>0,\,  B(x,\epsilon) \subset C\}$, where $B(x,\epsilon)$ denotes the open ball of radius $\epsilon$ centered at $x$. Let $\overline{\mathbb{R}}=\mathbb{R}\cup\{\pm\infty\}$.
A function $f\colon X\rightarrow\overline{\mathbb{R}}$ is convex if $f(\lambda x + (1-\lambda) y) \leq \lambda f(x) + (1-\lambda) f(y)$ for any $x, y \in X$ and $\lambda \in (0,1)$.
If the inequality holds strictly for all $x, y \in X$ such that $f(x)<\infty$ and $f(y)<\infty$ and $\lambda \in (0,1)$, then $f$ is strictly convex.
A function $f$ is proper if $f(x)>-\infty$ for all $x\in X$ and $f(x)<\infty$ for some $x\in X$.
We say that $f$ is closed if $\epi (f)=\{(x,t)\,|\,f(x)\le t\}\subseteq X\times \mathbb{R}$ is closed (in the standard product topology).
(See \cite[Proposition 2.5]{barbu2012} for equivalent definitions of closeness.) We say a function $f$ is CCP if it is convex, closed, and proper. Write $\dom f = \{x\in X\rvert \,  f(x)  < \infty\}$ to denote the (effective) domain of $f$.

\paragraph{Differentiability and subdifferentiability.}
Given $f\colon X \rightarrow \overline{\mathbb{R}}$, we define its convex conjugate $f^*\colon X^*\rightarrow \overline{\mathbb{R}}$ as $f^{*}(u) = \sup_{x \in X} \left\{ \inner{x}{u}-f(x)\right\}$. 
Define the biconjugate $f^{**}\colon X\rightarrow \overline{\mathbb{R}}$ (remember, $X^{**}=X)$ as
 $f^{**}(x) = \sup_{u \in X^*} \left\{\inner{u}{x}-f^*(u)\right\}$. 
If $f$ is CCP, then $f^{**}=f$ \cite[Theorem 2.22]{barbu2012}. We say that $f\colon X\rightarrow \overline{\mathbb{R}}$ is (Fr\'echet) differentiable at $x\in X$ if $|f(\cdot)|< \infty$ on an open neighborhood of $x$ and there exists a $\nabla f(x)\in X^*$ such that 
\[
\langle \nabla f(x),h\rangle=
\lim_{\varepsilon\rightarrow 0}
\frac{f(x+\varepsilon h)-f(x)}{\varepsilon},\qquad\forall\, h\in X.
\]
If the above holds, we define $\nabla f(x)\in X^*$ to be the gradient of $f$ at $x$. 
Write $\dom \nabla f$ for the set of points on which $f$ is Fr\'echet differentiable.
Finally, we say $f$ is differentiable (everywhere) if $f$ is differentiable for all $x\in X$. Let $f\colon X\rightarrow \overline{\mathbb{R}}$ be a CCP function. We say $g\in X^*$ is a subgradient of $f$ at $x$ if
\[
f(y)\ge f(x)+\langle g,y-x\rangle,\quad \forall\,y \in X.
\]
We write $\partial f(x)\subseteq X^*$ for the set of subgradients of $f$ at $x$. 

\paragraph{Smoothness.}
For $L>0$, we say $f\colon X\rightarrow \overline{\mathbb{R}}$ is $L$-smooth (with respect to $\norm{\cdot}$) if $|f(x)|<\infty$ everywhere (so $f\colon X\rightarrow \mathbb{R}$), $f$ is differentiable everywhere, and $\norm{ \nabla f(x)-\nabla f(y)}_{*} \leq L\norm{x-y}$ for all $x,y \in X$.
The following lemma establishes equivalent conditions defining $L$-smoothness.
\begin{lemma}\cite[Section~3.5]{zalinescu2002convex} [Banach--Baillon--Haddad theorem]\label{lemma:smooth}
Let $(X,\|\cdot\|)$ be a Banach space, $f\colon X\rightarrow \mathbb{R}$ be CCP and differentiable, and $L>0$.
The following are equivalent:
\begin{itemize}
\item[(i)] $\|\nabla f(x)-\nabla f(y)\|_*\le L\|x-y\|,\quad \forall\, x,y \in X$.
\item[(ii)] $\langle \nabla f(x)-\nabla f(y),x-y\rangle \ge \frac{1}{L}\|\nabla f(x)-\nabla f(y)\|^2_*,\quad \forall x,y \in X$.
\item[(iii)] $f(x)-f(y)+\inner{\nabla f(x)}{y-x}+\frac{1}{2L}\norm{\nabla f(x)-\nabla f(y)}^2_* \leq 0,\quad \forall x,y \in X$.
\end{itemize}
\end{lemma}
This result can be viewed as a generalization of the Baillon--Haddad theorem \cite[Corollary 10]{Baillon1977QuelquesPD} to Banach spaces, so we call it the \emph{Banach--Baillon--Haddad theorem}. We call the inequality of (\emph{iii}) the \emph{cocoercivity inequality}.

\paragraph{Strong convexity.}
For $\sigma>0$, we say $h\colon X\rightarrow \overline{\mathbb{R}}$ is \emph{$\sigma$-strongly convex} (with respect to $\norm{\cdot}$)
if
\[
h(y)\ge h(x)+\langle \partial h(x),y-x\rangle +\frac{\sigma}{2}\|y-x\|^2,\quad \forall\,x,y \in X,
\]
where the inequality is understood to hold for all elements of $\partial h(x)$. The following lemma establishes the equivalence between the strong convexity and the smoothness of the convex conjugate.

\begin{lemma}\cite[Section~3.5]{zalinescu2002convex} \label{lemma:conjugate_of_sc_is_smooth}
Let $(X,\|\cdot\|)$ be a reflexive Banach space, $h\colon X\rightarrow\overline{\mathbb{R}}$ a CCP function, and $\sigma>0$.
Then,
$h$ is $\sigma$-strongly convex with respect to $\norm{\cdot}$ if and only if
$h^{*}$ is $1/\sigma$-smooth with respect to $\norm{\cdot}_{*}$. 
\end{lemma}

\paragraph{Bregman divergence.}
Let $(X,\|\cdot\|)$ be a Banach space and let $h\colon X \rightarrow \overline{\RR}$. The \emph{Bregman divergence} with respect to $h$ is $D_h \colon X\times X\rightarrow\overline{\RR}$ defined as 
\[
D_h(x,x')=
\left\{
\begin{array}{ll}
h(x)-h(x')-\inner{\nabla h(x')}{x-x'}&\text{ for }x \in \dom h,\,x'\in \dom \nabla h\\
\infty&\text{ otherwise.}
\end{array}\right.
\]
The standard inequalities for differentiable convex $f$ can be written using the Bregman divergence as follows:
If $f$ is convex,
\begin{equation}
D_f(x,x') \geq 0, \qquad \forall\,x,x'\in X. \label{eqn::convex_with_bregman} \tag{cvx-ineq}
\end{equation}
If $f$ is convex and $L$-smooth, by Lemma~\ref{lemma:smooth},
\begin{equation}
D_f(x,x') -  \frac{1}{2L}\norm{\nabla f(x)-\nabla f(x')}^2_{*}\geq 0,\qquad \forall\,x,x'\in X.\label{eqn::coco_with_bregman} \tag{coco-ineq}
\end{equation}
If $y \in \dom \nabla h^*$, the Fenchel inequality \cite{fenchel1949conjugate} is
\[
D_h(x,\nabla h^*(y)) = h(x) + h^*(y ) - \inner{y}{x}
\geq 0, \qquad \forall\,x\in X,\,y\in X^*. \label{eqn::fenchel} \tag{fenchel-ineq}
\]

\paragraph{Legendre convexity.} We follow the definitions from \cite{Bauschke2001ESSENTIALSE}.
Let $h\colon X \to \bar{\mathbb{R}}$
be a CCP function. $h$ is \emph{essentially smooth} if $\partial h$ is locally bounded and single-valued on its domain.
Additionally, if $( \partial h)^{-1}$ is locally bounded on its domain and $h$ is strictly convex on every convex subset of $\dom \partial h$, we say $h$ is \emph{Legendre convex}.  To clarify, a set-valued function is locally bounded on $Y$ if, for any $ y \in Y$, there is a neighborhood of $y$ in $Y$ on which the output values of the set-valued function are bounded.

\subsection{Prior works}
\label{sec:prior_works}
\paragraph{First-order methods reducing gradient norm.}
Consider minimizing a differentiable $L$-smooth function $f\colon X\rightarrow\mathbb{R}$ with $N$ iterations of a first-order method. Assume $x_\star \in \argmin_{x\in \mathbb{R}^d} f(x)$ exists. Write $f(x_\star)=f_\star$. Classically, plain gradient descent achieves a $\mathcal{O}(L(f(x_0)-f_\star)/N )$ rate on $\|\nabla f(\cdot)\|_2^2$ for $L$-smooth functions, both convex and non-convex. Nemirovsky's classical work \cite{nemirovsky1991optimality,nemirovsky1992information} provides 
lower bounds for reducing $\norm{\nabla f(x_N)}^2$ via first-order deterministic methods. They proved there exists a convex and $L$-smooth function $f$ such that no $N$-step first-order deterministic method can reduce \( \| \nabla f(x_N) \|^2 \) faster than \( \Omega \left( L(f(x_0) - f_\star)/N^2 \right) \), with respect to the worst-case gaurantee. The same argument holds for \( \Omega \left( L^2 \| x_0 - x_\star \|_2^2/N^4 \right) \).

The formal study of first-order methods for efficiently reducing the gradient magnitude $\norm{\nabla f(x)}_2^2$ was initiated by Nesterov \cite{nesterov2012make}, where he proposed the regularization technique to obtain a $\tilde{\mathcal{O}}(L^2\|x_0-x_\star\|_2^2/N^4)$ rate, where $\tilde{\mathcal{O}}$ ignores logarithmic factors. This regularization technique was also used in the stochastic setup \cite{zhu2018how}. The logarithmic gap of Nesterov was closed by Kim and Fessler, who proposed OGM-G and established a $\mathcal{O}(L(f(x_0)-f_\star)/N^2 )$ rate \cite{kim2021optimizing}; it was pointed out in \cite[Remark 2.1]{nesterov2020primal} that OGM-G's rate can be combined with the classical Nesterov acceleration to imply the optimal $\mathcal{O} (L^2\|x_0-x_\star\|_2^2/N^4)$ rate. Kim and Fessler's discovery, which was made with the computer-assisted proof methodology called the performance estimation problem \cite{drori2014performance,taylor2017smooth,taylor2017exact,kim2016optimized}, kickstarted a series of follow-up works: OBL-G, a variant of OGM-G that allows a backtracking linesearch \cite{park2021optimal}; M-OGM-G, which attains the same $\mathcal{O}(L(f(x_0)-f_\star)/N^2 )$ rate with simpler rational coefficients \cite{ZhouTianSoCheng2021_practical}; and better human-understandable analyses using potential-functions \cite{lee2021geometric,diakonikolas2021potential}. On the other hand, \cite{lan2023optimal} utilizes the regularization technique of Nesterov in a clever aggregate analysis way %an accumulative fashion 
to remove the logarithmic terms and achieves the optimal rate without using the concatenation technique that we describe in Section~\ref{sec::concat}. Furthermore, \cite{lan2023optimal} accomplishes this adaptively in the sense that the method does not require a priori knowledge of $L$.

There are also follow-up works in other setups; continuous-time analysis of OGM-G \cite{Suh2022continuous} and the proximal variants FISTA-G \cite{lee2021geometric} and SFG \cite{kim2023timereversed}. However, despite these active works, the acceleration mechanisms behind reducing gradient magnitude are understood much less compared to the classical techniques for reducing the function value.

\paragraph{H-duality.} Interestingly, OGM-G exhibits a certain symmetry against OGM \cite{drori2014performance,kim2016optimized}, an accelerated first-order method that improves upon Nesterov's acceleration by a factor of $2$ for reducing function values of smooth convex functions. The recent work \cite{kim2023timereversed} showed that this symmetry is fundamentally due to H-duality, a one-to-one correspondence between first-order methods that reduce function value and those that reduce gradient norm. 

We quickly review H-duality. Let $X$ be an Euclidean space equipped with the standard Euclidean norm. Let $f\colon X\rightarrow\reals$ be convex and $L$-smooth. Consider the problem 
\begin{align*}
\begin{array}{ll}
\underset{x \in X}{\mbox{minimize}}&\quad f(x).
\end{array}
\end{align*} 
Consider an $N$-step fixed-step first-order method (FSFOM) with step sizes $\{h_{k,i}\}_{0\leq i < k \leq N}\subset \RR$, defined by:
\begin{align*}
    x_{k+1} = x_k - \frac{1}{L}\sum_{i=0}^{k} h_{k+1,i}\nabla f(x_i), \quad k=0,1,\dots, N-1,
\end{align*}
where $x_0\in X$ is a starting point. We say this FSFOM is defined by the lower triangular matrix $H\in\RR^{N \times N}$ containing  $\{h_{k,i}\}_{0\leq i < k \leq N}$, i.e., $H_{k+1,i+1}=h_{k+1,i}$ if $0\leq i\leq k\leq N-1$, and $H_{k+1,i+1}=0$ otherwise. For $H\in\mathbb{R}^{N\times N}$ define its \emph{anti-diagonal transpose} $H^A\in \mathbb{R}^{N\times N}$ with $H^A_{i,j}=H_{N-j+1,N-i+1}$ for $i,j=1,\dots,N$. We refer to the FSFOM defined by $H^A$ as the \emph{H-dual} of the FSFOM defined by $H$.

To clarify, the H-dual is a dual of a first-order \emph{method}, an algorithm. This contrasts with the usual notions of duality between primal and dual spaces, functions, or optimization problems. The H-dual operation is defined mechanically, but the H-duality theorem, which we soon state as Fact~\ref{prop:H-duality}, establishes a non-trivial property between the H-dual FSFOM pair.

\begin{fact}\label{prop:H-duality}
\emph{[H-duality theorem, Informal \cite[Theorem~1]{kim2023timereversed}]}
Let $\cA$ and $\cB$ be FSFOMs that are H-duals of each other and $\alpha>0$. Then the following are equivalent.
\begin{itemize}
    \item $f(x_N)-f_\star\leq \frac{\alpha L}{2}\norm{x_0-x_\star}^2$ for $\cA$ can be proved with primal energy function structure.
    \item $\frac{1}{2L}\norm{\nabla f(x_N)}^2 \leq \alpha\left(f(x_0)-f_\star\right)$ for $\cB$ can be proved with dual energy function structure.
\end{itemize}
\end{fact}
We will clarify the meaning of  `primal and dual energy function structures' in Section~\ref{ss::mirror_duality} along with our main results. The H-duality theorem provides a sufficient condition that ensures an FSFOM defined by $H\in\RR^{N\times N}$ with a convergence guarantee on $(f(x_N)-f_{\star})$ can be ``H-dualized'' to obtain an FSFOM defined by $H^A\in\RR^{N\times N}$ with a convergence guarantee on $\|\nabla f(x_N)\|^2$, and vice versa. 

\paragraph{Instances of H-duality.}  OGM \cite{kim2016optimized} is an FSFOM with guarantee
\[
f(x_N)-f_\star\le \frac{1}{\zeta_N^2}\frac{L\norm{x_0-x_\star}_2^2}{2},
\]
and OGM-G \cite{kim2021optimizing} is an FSFOM with guarantee
\[
\frac{1}{2L}\|\nabla f(x_N)\|^2_2\le \frac{1}{\zeta_N^2}\left(f(x_0)-f_\star \right),
\]
where $\{\zeta_i\}_{i=0}^{N}$ is defined by $\zeta_0=1$, $\zeta_{i+1}^2 - \zeta_{i+1} =\zeta_i^2$ for $i=0,\dots,N-2$ and $\zeta_{N}^2 - \zeta_N = 2\zeta_{N-1}^2$. It turns out that OGM and OGM-G are H-duals of each other, and their guarantees imply each other through the H-duality theorem (with $\alpha=1/\zeta_N^2$). 

As another example, the standard gradient descent, $x_{k+1}=x_k-\frac{1}{L}\nabla f(x_k)$ for $k=0,\dots,N-1$, has guarantees
\[
f(x_N)-f_\star\le \frac{1}{2N+1}\frac{L\norm{x_0-x_\star}^2_2}{2}\]
and
\[
\frac{1}{2L}\|\nabla f(x_N)\|^2_2\le \frac{1}{2N+1}\left(f(x_0)-f_\star \right).
\]
Gradient descent (with constant stepsize) is ``self-H-dual'', i.e., the H-dual of gradient descent is gradient descent itself, and these two guarantees imply each other through the H-duality theorem (with $\alpha=1/(2N+1)$).
 
Recently, \cite{DasGupta2024} numerically observed that gradient-descent type methods ($x_{i+1}=x_i-\frac{h_i}{L}$ for $i=0,\dots,N-1$) with appropriately chosen step sizes can exhibit a faster convergence rate than $\mathcal{O}(1/N)$. In a series of works \cite{grimmer2023accelerated,grimmer2024accelerated,grimmer2024provably,altschuler2023acceleration1,altschuler2023acceleration2}, the rate $\mathcal{O}(1/N^{1.2716})$ was theoretically established. The latest of this line of work \cite{grimmer2024accelerated}, which presents the guarantee with the best constant, also presents an H-duality phenomenon: If $N=2^n-1$ for some $n \in \mathbb{N}$, the method $x_{i+1}=x_i-\tfrac{h_i^{\mathrm{(left)}}}{L} \nabla f(x_i)$ with $\{h_i^{\mathrm{(left)}}\}_{i=0}^{N-1}$ has a guarantee
\[
f(x_N)-f_\star \le \frac{1}{r_N}\frac{L\norm{x_0-x_\star}^2_2}{2},
\]
where $r_N \approx N^{\log_2(1+\sqrt{2})}$, and its H-dual method, $x_{i+1}=x_i-\tfrac{h_{N-1-i}^{\mathrm{(left)}}}{L} \nabla f(x_i)$, has a guarantee
\[
\frac{1}{2L}\|\nabla f(x_N)\|^2_2 \le \frac{1}{r_N}\left(f(x_0)-f_\star \right).
\]
Although this result is not a direct instance of the existing H-duality theorem, the symmetry points to the existence of a broader H-duality phenomenon beyond what is elucidated in \cite[Theorem~1]{kim2023timereversed}.

\paragraph{Mirror-descent-type methods.}
The framework of mirror descent has been a cornerstone for designing efficient methods for optimization problems with structures that are well-behaved in non-Euclidean geometries \cite{nemirovsky1983problem,beck2003mirror,ben2001ordered,Juditsky20105FO,Censor1992proximal,Teboulle1992proximal,Jonathan1993nonlinear}. A mirror-descent-type method uses a distance-generating function $\phi\colon X\rightarrow\overline{\mathbb{R}}$, which gives rise to the Bregman divergence $D_\phi(\cdot,\cdot)$, for minimizing a convex function $f\colon X\rightarrow\overline{\mathbb{R}}$. Consider $N$ iterations of such a method.

The standard mirror descent method attains a $\cO(D_\phi(x_\star,x_0) /\sqrt{N})$ rate on the best iterate for nonsmooth convex $f$ with bounded subgradients and a standard set of assumptions.
For differentiable convex $f$, 
 \cite{doi:10.1287/moor.2016.0817,doi:10.1137/16M1099546,VanNguyen2017} presented the relative smoothness condition, which relaxes the strong convexity of distance generating function, and achieved a $\cO(D_\phi(x_\star,x_0) /N)$ rate on the last iterate under relative smoothness. Under a set of stronger assumptions, accelerating the rate on function value is possible: the Improved Gradient Algorithm achieves a $\cO(D_\phi(x_\star,x_0) /N^2)$ rate for constrained smooth optimization problems within Euclidean spaces \cite{auslender2006interior};
 % \jd{are you sure this is the right reference?}; 
 with a uniformly convex distance-generating function a faster rate than $\cO(D_\phi(x_\star,x_0) /N)$ 
 can be achieved with convex functions that are smooth with respect to the $\ell_p$-norm \cite{alexandre2018optimal};
a method using an auxiliary regularizer function \cite{tseng2008}; and a method proposed from an ODE perspective on the linear coupling \cite{krichene2015amd}. The prior accelerated mirror-descent-type method most relevant to our work is the accelerated Bregman proximal gradient method, which attains a $\cO(D_\phi(x_\star,x_0) /N^2)$ rate on function values, under the setup of smooth $f$ and strongly convex $\phi$ with respect to any norm $\|\cdot\|$ \cite{d2021acceleration}. 

Recently, connections between various sampling algorithms and entropic mirror descent have been established, which utilize negative entropy as a distance-generating function: \cite{chopin2023connection} showed the equivalence with tempering and \cite{karimi2023sinkhorn} demonstrated the connection with generalized Sinkhorn iterations.

\paragraph{Mirror-descent-type methods reducing gradient magnitude. }
 There has been some recent attention toward mirror-descent-type methods that efficiently reduce gradient magnitude. Dual preconditioned gradient descent \cite{doi:10.1137/19M130858X} efficiently reduces the gradient measurement under the relative smoothness condition in the dual space, a condition later characterized in \cite{doi:10.1137/21M1465913}. This method was extended to the non-convex and composite setups \cite{laude2023anisotropic}. On the other hand, \cite{diakonikolas2023complementary} used a regularization technique to develop methods to reduce the magnitude of the gradient, measured by the dual norm. In Section~\ref{sec::disc_md}, we review dual preconditioned gradient descent in further detail as it serves as the foundation for the development of \ref{eqn::AMD-G}.

\subsection{Contributions and organization}

\paragraph{Contributions.} 
This work presents two main contributions, one conceptual and one algorithmic. The first main contribution is mirror duality, presented as Theorem~\ref{thm::mirror_duality_main}. The H-duality presented in \cite{kim2023timereversed} made progress by establishing a duality principle between the task of reducing function value and the task of reducing gradient magnitude. Mirror duality extends H-duality to the mirror descent setup and provides a further richer understanding.

Our second main contribution is the method \ref{eqn::AMD-G} and its convergence analysis, stated as Corollary~\ref{thm::AMD-G}. To the best of our knowledge, \ref{eqn::AMD-G} is the first mirror-descent-type method for reducing gradient magnitude at an accelerated rate. The generality of \ref{eqn::AMD-G} accommodates optimization problems with structures that are well-behaved in non-Euclidean geometries and allows us to measure the gradient magnitude with a smooth convex function of our choice, not necessarily a norm. The applications of Section~\ref{s:applications} illustrate the practical utility of \ref{eqn::AMD-G}.

\paragraph{Organization.} 
Section~\ref{sec::disc_md} reviews \ref{eqn::MD} and \ref{eqn::dual_MD} and points out a symmetry between these two mirror-descent-type methods. 
Section~\ref{sec::amd} reviews \ref{eqn::AMD}, an accelerated mirror-descent-type method that reduces the function value. The specific structure of the convergence analysis of \ref{eqn::AMD} is essential for obtaining our main result.

Section~\ref{ss:cfom} quickly establishes some definitions.
Section~\ref{ss::mirror_duality} formally presents the mirror duality theorem, our first main contribution. Section~\ref{ss:dual-amd} presents the \ref{eqn::AMD-G} method, our second main contribution, obtained by applying the mirror duality theorem to \ref{eqn::AMD}. Section~\ref{s:applications} presents applications illustrating the practical utility of \ref{eqn::AMD-G}. Section~\ref{ss:h-duality-relation} compares the prior results on H-duality with our generalized framework of mirror duality.
Finally, Section~\ref{s:conclusion} concludes the paper and provides some discussion on future directions.

\section{Review of dual mirror descent and accelerated mirror descent}
\subsection{\ref{eqn::MD} and \ref{eqn::dual_MD}}
\label{sec::disc_md}

We provide a brief overview of mirror descent (\ref{eqn::MD}) and dual-mirror descent (\ref{eqn::dual_MD}). The relationship between these two methods presents the following insight: Let $f$ be the objective function and $\phi$ be the distance-generating function. Then, \ref{eqn::MD} is a method finding a $y\in X^*$ minimizing $f(\nabla \phi^*(y))$ and by swapping the roles of $f$ and $\phi^*$, we get \ref{eqn::dual_MD}, which finds an $x\in X $ minimizing $\phi^*(\nabla f(x))$. 

\paragraph{Mirror descent.} Let $X$ be a reflexive Banach space, $C \subseteq X$ be a nonempty closed convex set, and  $f\colon X \to \overline{\mathbb{R}}$ be a Legendre convex function. To state \ref{eqn::MD}, consider the problem
\begin{align*}
\begin{array}{ll}
\underset{x \in C}{\mbox{minimize}}&\quad f(x).
\end{array}
\end{align*} 
Assume that $x_\star \in \argmin_{x \in C} f(x) $ exists. Given an approximate solution $x$, we measure its suboptimality via
\[
f(x)-f(x_\star).
\]
To obtain $x\in C$ with small $f(x)-f(x_\star)$, consider a Legendre convex function $\phi\colon X\to\overline{\mathbb{R}}$ such that $\overline{\dom \phi}=C$. For $\alpha>0$ and any starting point $y_0\in \inte \dom \phi^*$, the Mirror Descent (MD) method is 
\begin{equation}
\begin{aligned}
y_{k+1} &= y_k - \alpha\nabla f(x_k)\\
x_{k+1} &= \nabla \phi^{*}(y_{k+1}),
     \end{aligned}
\tag{MD} \label{eqn::MD}
\end{equation}
for $k=0,1,\dots, N-1$, and $x_0=\nabla  \phi^*(y_0)$. We view $x_N$ as the output of \ref{eqn::MD}.
Consider the following additional assumption regarding $f$ and $\phi$. 
\begin{itemize}
    \item[(i)] For some $\lambda>0$, $\lambda\phi-f$ is convex on $\inte \dom \phi$. 
    \item[(ii)] $\inte \dom \phi \subseteq \inte \dom f$.
\end{itemize}
Under the above assumption, we have the following convergence guarantee for $x_N$.
\begin{fact}\cite[Theorem~1]{Dragomir2021}
\label{fact:md-bound}
Let $X$ be a reflexive Banach space and $f,\phi \colon X \to \overline{\mathbb{R}}$ Legendre convex functions. Then the iterates of \ref{eqn::MD} are well defined. Assume (i) and (ii). If we further assume that $x_\star \in\dom \phi$, then for $\alpha\in (0,\lambda]$,
\begin{align*}
    f(x_N)-f(x_\star) \leq \frac{1}{\alpha N}D_\phi(x_\star ,x_0).
\end{align*}
\end{fact}

\paragraph{Dual mirror descent.} Consider the problem
\begin{align*}
\begin{array}{ll}
\underset{x \in \dom \nabla f}{\mbox{find}}&\quad 0=\nabla f(x).
\end{array}
\end{align*} 
This problem is equivalent to the previous one,  $\min_{x\in C}f(x)$, if $C=X$ and $f$ is differentiable ($\dom \nabla f =X$). Given an approximate solution $x$ for this problem, we measure its suboptimality via
\[
\psi^*(\nabla f(x)),
\]
where $\psi^*\colon X^*\rightarrow \overline{\mathbb{R}}$ is a Legendre convex function
such that $\psi^*(0)=0$ and $0\in X^*$ is the unique minimizer of $\psi^*$. So, $\psi^*$ quantifies the magnitude of $\nabla f(x)$.\footnote{Every stationary point is a global function minimum in convex optimization, so $\psi^*(\nabla f(x))=0$ is equivalent to $f(x)-\inf_{x\in X}f(x)=0$. However, [$\psi^*(\nabla f(x))$ being small] is not equivalent to [$f(x)-\inf_{x\in X}f(x)$ being small] in the sense that for any $\varepsilon>0$ and $\psi^*$, there is an $f$ such that $\psi^*(\nabla f(x))<\varepsilon$ but $f(x)-\inf_{x\in X}f(x)>1$.
Even though the solution set is unchanged, changing the suboptimality measure leads to different efficient methods for finding approximate solutions.} Dual preconditioned gradient descent \cite{doi:10.1137/19M130858X}, formulated as follows, efficiently reduces $\psi^*(\nabla f(x))$. For $\alpha>0$, starting point $q_0 \in\inte \dom f$ and $r_0=\nabla f(q_0)$,
\begin{equation}
\begin{aligned} 
     q_{k+1} &= q_k - \alpha\nabla \psi^{*}(r_k)\\
     r_{k+1} &= \nabla f(q_{k+1}),
\end{aligned}
\tag{dual-MD} \label{eqn::dual_MD}
\end{equation}
for $k=0,\dots,N-1$. We will refer to this method as \ref{eqn::dual_MD}. The reason for this alternative naming will be made clear in Section~\ref{sec::mirror_duality}. 

The output of \ref{eqn::dual_MD} is $q_N$. Consider the following additional assumption regarding $f$ and $\psi$. 
\begin{itemize}
    \item[(i)'] For some $\tau>0$, $\tau f^*-\psi^*$ is convex on $\inte \dom f^*$. 
    \item[(ii)'] $\inte \dom f^* \subseteq \inte\dom \psi^*$.
\end{itemize}
Under this assumption, we have the following convergence guarantee for $q_N$.
\begin{fact}\cite[Theorem~3.9]{doi:10.1137/19M130858X}
\label{cor:dual-md} Let $X$ be a reflexive Banach space and $f,\psi \colon X \to \overline{\mathbb{R}}$ be Legendre convex functions. Assume (i)' and (ii)'. Then, the iterates of \ref{eqn::dual_MD} are well defined. 
If we further assume that $\psi^*(0)=0$ and $0\in X^*$ is the unique minimizer of $\psi^*$, then for $\alpha\in (0,\tau]$,
    \[
   \psi^*(\nabla f(q_N))=   \psi^*(r_N) \leq
   \frac{1}{\alpha N}\big(f(q_0)-\inf_{x\in X}f(x)\big). 
    \]
    (If $\inf_{x\in X}f(x)=-\infty$, then the right-hand-side is $+\infty$ and the bound vacuously holds.)
\end{fact}
 
Interestingly, \ref{eqn::dual_MD} can be obtained by swapping $f$ with $\phi^*=\psi^*$ in \ref{eqn::MD}, and this symmetry between them leads to a crucial observation. 
 In \ref{eqn::dual_MD}, $\psi^*=\phi^*$ serves as the objective function, so the reduction of $\psi^*(\nabla f(q_N))$ in \ref{eqn::dual_MD} as guaranteed by Fact~\ref{cor:dual-md} is really an immediate consequence of \ref{eqn::MD} reducing $f(x_N)=f(\nabla \phi^*(y_N))$ as guaranteed by Fact~\ref{fact:md-bound}.
At this point, \ref{eqn::MD} and \ref{eqn::dual_MD} exhibit a certain symmetry, and one method reduces $f(\cdot)-f(x_\star)$ while the other ``dual'' method reduces $\psi^*(\nabla f(\cdot))$. In Section~\ref{sec::amd}, we state \ref{eqn::AMD}, which reduces $f(\cdot)-f(x_\star)$ at an $\mathcal{O}\left(1/N^2\right)$ rate under stronger assumptions than needed for \ref{eqn::MD}. This leads to a natural question: Can we ``dualize'' \ref{eqn::AMD} to obtain a mirror-descent-type method with an $\mathcal{O}\left( 1/N^2 \right)$ rate on $\psi^*(\nabla f(\cdot))$? 
For reasons that we explain later in Section~\ref{sec::mirror_duality}, merely interchanging the roles of $f$ and $\phi^*=\psi^*$ does not work for obtaining the dual counterpart of \ref{eqn::AMD}. The answer is obtained through \emph{mirror duality} that we present in Section~\ref{sec::mirror_duality}.

The versions of Facts~\ref{fact:md-bound} and \ref{cor:dual-md} presented in \cite{Dragomir2021} and \cite{doi:10.1137/19M130858X} rely on slightly weaker assumptions than we state. We make this simplification to focus on our goal of introducing the concept of swapping $f$ and $\phi^* = \psi^*$ in mirror-descent-type methods.

\paragraph{Connection with dual averaging.} Our formulation of \ref{eqn::MD} is closely related to dual averaging \cite{nesterov2009primal}. (In dual averaging, \emph{dual} refers to the duality between the primal $X$ and dual $X^*$ spaces, while in our 
\ref{eqn::dual_MD}, \emph{dual} refers to mirror duality.) Dual averaging and mirror descent are equivalent in the so-called unconstrained setup, which is what we consider. In the constrained setup, the optimization problem is constrained to a further subset of $C=\overline{\dom \phi}$, and in this constrained setup, mirror duality and dual averaging differ. The dual algorithms of our work and the notion of mirror duality do not seem to immediately generalize to the constrained setup, and studying this extension would be an interesting direction of future work.
\subsection{Accelerated Mirror Descent}
\label{sec::amd}
Again, consider the problem
\begin{align*}
\begin{array}{ll}
\underset{x \in C}{\mbox{minimize}}&\quad f(x),
\end{array}
\end{align*} 
where $X$ is a reflexive Banach space, $C\subseteq X$ is a nonempty closed convex set and $f\colon X\to\RR$ is a CCP function. Assume $x_\star \in \argmin_{x \in C} f(x) $ exists. Given an approximate solution $x$, we measure its suboptimality via $f(x)-f(x_\star)$. We use a CCP function $\phi\colon X\to\overline{\RR}$ such that $\overline{\dom \phi}=C$ for the distance-generating function.
Assume that
\begin{itemize}
     \item [\hypertarget{iiii}{(\hyperlink{iiii}{A1})}]
     $f$ is $L$-smooth; $\phi$ is $\sigma$-strongly convex; $L>0$, $\sigma>0$.
\end{itemize}

Let $N$ be a positive integer. Let $\{\theta_i\}_{i = -1}^N\subset \mathbb{R}$ be a non-decreasing sequence satisfying
\begin{align*}
    \theta_{-1}=0,\qquad
    \theta_{0}=1,\qquad\theta_{i}^2-\theta_{i} \leq \theta_{i-1}^2\,\, \text{for} \,\, 1\leq i\leq N-1,\qquad\theta_N =\theta_{N-1}.
\end{align*}
When $\theta_{i}^2-\theta_{i} =\theta_{i-1}^2$ for $i=1,\dots,N-1$, the sequence becomes exactly the same as the $\theta$-sequence of Nesterov's fast gradient method \cite{10029946121}, except for the final $\theta_N$. The variant of accelerated mirror descent \eqref{eqn::AMD} we consider is
    \begin{align}
        & y_{k+1} = y_k - \frac{\sigma}{L}\left(\theta_{k}^2-\theta_{k-1}^2\right) \nabla f(x_k)
    \label{eqn::AMD} \tag{AMD} \\
        & x_{k+1}=\frac{\theta_{k}^2}{\theta_{k+1}^2}x_k + \frac{\theta_{k+1}^2 -\theta_k^2}{\theta_{k+1}^2}\nabla \phi^{*}(y_{k+1}) +  \frac{\theta_k^2 - \theta_{k-1}^2}{\theta_{k+1}^2}\left(\nabla \phi^{*}(y_{k+1})-\nabla \phi^{*}(y_{k})\right)\nonumber
    \end{align}
for $k=0,1,\dots,N-1$, where $y_0 \in X^*$ is any starting point and $x_0=\nabla \phi^*(y_0)$.  We view $x_N$ as the output of \ref{eqn::AMD}.

\begin{fact}
\label{fact:amd}
Let $X$ be a reflexive Banach space and let $f\colon X\to\RR$ and $\phi\colon X\to\RR$ be CCP.%
\footnote{Our analysis does not require the distance-generating function $\phi$ to be differentiable. Prior work such as \cite{xiao2010dual} has also carried out mirror-descent-type analyses under non-differentiability of $\phi$.}
Assume \hyperlink{iiii}{(A1)}. Then, the iterates of \ref{eqn::AMD} are well defined and 
$x_k\in C$ for $k=0,\dots,N$. Furthermore, if $x\in \dom \phi$,
\begin{align*}
    f(x_N) - f(x) \leq \frac{LD_\phi(x,x_0) }{\sigma\theta_N^2}
       \stackrel{(\bullet)}{=}\mathcal{O}\left(\frac{LD_\phi(x,x_0)}{\sigma N^2}\right)
\end{align*}
where  $(\bullet)$ holds if $\theta_{i}^2-\theta_{i}= \theta_{i-1}^2$ for $0 \leq i\leq N-1$.
\end{fact}

\begin{proof}
First, we show $x_k$ is a convex combination of $\{\nabla \phi^*(y_0),\dots,\nabla \phi^*(y_k)\}$. Then the fact $\nabla \phi^*\colon X^* \rightarrow C$ yields $x_k \in C$. For $k=0,\dots,N-1$,
\begin{align*}
   x_{k+1}&=\frac{\theta_{k}^2}{\theta_{k+1}^2}x_k + \frac{\theta_{k+1}^2 -\theta_{k-1}^2}{\theta_{k+1}^2}\nabla \phi^{*}(y_{k+1}) -  \frac{\theta_k^2 - \theta_{k-1}^2}{\theta_{k+1}^2} \nabla \phi^*(y_k) \\
    &= \frac{\theta_{k-1}^2}{\theta_{k+1}^2}x_{k-1} + \frac{\theta_{k+1}^2 -\theta_{k-1}^2}{\theta_{k+1}^2}\nabla \phi^{*}(y_{k+1}) + \frac{\theta_{k-1}^2 -\theta_{k-2}^2}{\theta_{k+1}^2}\nabla \phi^{*}(y_{k})\\
    &\quad -  \frac{\theta_{k-1}^2- \theta_{k-2}^2}{\theta_{k+1}^2} \nabla \phi^*(y_{k-1})
    \\
    &\qquad \qquad \qquad \qquad \vdots
    \\
    &= \frac{\theta_{k+1}^2 -\theta_{k-1}^2}{\theta_{k+1}^2}\nabla \phi^{*}(y_{k+1})  + \sum_{s = 1}^k  \frac{\theta_{s-1}^2 -\theta_{s-2}^2}{\theta_{k+1}^2}\nabla \phi^{*}(y_{s}).
\end{align*}
Now we provide the convergence analysis of \ref{eqn::AMD}.
\paragraph{Convergence analysis.}
First, define $\{u_i\}_{i=0}^{N}$ as $u_i=\frac{\sigma}{L}\theta_i^2$. Also define an energy function $\{\mathcal{U}_k\}_{k=0}^{N}$ recurrently as $\cU_0 = D_\phi(x,x_0) - \frac{\sigma}{L}\theta_{0}^2D_f(x,x_0)$ and
\begin{align*}
\quad \mathcal{U}_{k+1} = &\; \mathcal{U}_k - (u_{k+1}-u_k)D_f\left(x_k , x_{k+1}\right)\\
&- u_{k}\Big(
\underbrace{D_f\left(x_k,x_{k+1}\right)-\frac{1}{2L}\norm{\nabla f(x_k)-\nabla f(x_{k+1})}_{*}^2 }_{\ge 0\text{ by \eqref{eqn::coco_with_bregman}}}\Big) \\
 &\quad - \Big(
 \underbrace{D_{\phi^*}\left(y_k , y_{k+1} \right)-\frac{\sigma}{2}\norm{\nabla \phi^{*}(y_{k+1})-\nabla \phi^*(y_{k})}^2 }_{\ge 0\text{ by \eqref{eqn::coco_with_bregman}}}\Big)
\end{align*}
for $k=0,\dots,N-1$. Due to \eqref{eqn::convex_with_bregman} and \eqref{eqn::coco_with_bregman}, $\{\cU_{k}\}_{k=0}^{N}$ is non-increasing . If we can prove  $f(x_N) - f(x) \leq \frac{L}{\sigma \theta_N^2} \cU_N$,
the monotonicity of $\{\cU_{k}\}_{k=0}^{N}$ provides the convergence rate as follows:
\begin{align*}
  f(x_N)-f(x) \leq  \frac{L}{\sigma \theta_N^2}\cU_N\leq \dots  \leq \frac{L}{\sigma \theta_N^2}\cU_{0}  \leq \frac{L}{\sigma \theta_N^2}D_\phi(x,x_0).
\end{align*}
Now, the only  part left is proving the inequality $f(x_N) - f(x) \leq \frac{L}{\sigma \theta_N^2} \cU_N$. Define $\mathbf{U}_{\text{AMD}}$, which is the function of algorithm parameters and iterates of \ref{eqn::AMD}, via
\begin{align*}
    \mathcal{U}_N &= \frac{\sigma}{L}\theta_N^2 \left(f(x_N)-f(x) \right) + \underbrace{\phi(x) + \phi^*(y_N)-\inner{y_N}{x}}_{\geq 0 \text{ by \eqref{eqn::fenchel}}} \\
    & \quad +\inner{y_N-y_0+\sum_{i=0}^{N}\left(\frac{\sigma}{L}\theta_i^2-\frac{\sigma}{L}\theta_{i-1}^2\right)\nabla f(x_i)}{x }+ \mathbf{U}_{\text{AMD}}.
\end{align*}
Here we used $D_\phi(x,x_0)=\phi(x)+\phi^*(y_0)-\inner{y_0}{x}$. Observe that  
\begin{align*}
& y_N - y_0 + \sum_{i=0}^{N}\left(\frac{\sigma}{L}\theta_i^2-\frac{\sigma}{L}\theta_{i-1}^2\right) \nabla f(x_i) \\
= & -\sum_{k=0}^{N-1} (y_k - y_{k+1}) + \sum_{i=0}^{N}\left(\frac{\sigma}{L}\theta_i^2-\frac{\sigma}{L}\theta_{i-1}^2\right) \nabla f(x_i) \\
= & - \sum_{k=0}^{N-1} \frac{\sigma}{L}\left(\theta_{k}^2 - \theta_{k-1}^2 \right)\nabla f(x_k) + \sum_{i=0}^{N}\left(\frac{\sigma}{L}\theta_i^2-\frac{\sigma}{L}\theta_{i-1}^2\right) \nabla f(x_i) \\
= & 0.\quad \left(\because \theta_N = \theta_{N-1} \right).
\end{align*}
Therefore, it is enough to show $\mathbf{U}_{\text{AMD}} \geq 0$. We simplify $\mathbf{U}_{\text{AMD}}$ as follows: 
\begin{align*}
   &\; \mathbf{U}_{\text{AMD}}\\
   = &\frac{\sigma}{2L^2}\sum_{k=0}^{N-1} \theta_k^2 \norm{\nabla f(x_k) -\nabla f(x_{k+1})}^2_{*} + \frac{\sigma}{2} \sum_{k=0}^{N-1} \norm{\nabla \phi^*(y_k) -\nabla \phi^*(y_{k+1})}^2 \\
   & \quad + \sum_{k=0}^{N-1} \inner{y_k - y_{k+1}}{\nabla \phi^*(y_{k+1})}  - \frac{\sigma}{L}\sum_{k=0}^{N} \theta_k^2 \inner{\nabla f(x_k)-\nabla f(x_{k+1})}{x_k} \\
   = & \frac{\sigma}{2L^2}\sum_{k=0}^{N-1} \theta_k^2 \norm{\nabla f(x_k) -\nabla f(x_{k+1})}^2_{*} + \frac{\sigma}{2} \sum_{k=0}^{N-1} \norm{\nabla \phi^*(y_k) -\nabla \phi^*(y_{k+1})}^2 \\
   & \quad + \frac{\sigma}{L}\sum_{k=0}^{N-1} \inner{\left(\theta_k^2 - \theta_{k-1}^2 \right)\nabla f(x_k)}{\nabla \phi^*(y_{k+1})} - \frac{\sigma}{L}\sum_{k=0}^{N} \inner{\nabla f(x_k)}{\theta_k^2 x_k - \theta_{k-1}^2 x_{k-1}} \\
   = & \frac{\sigma}{2L^2}\sum_{k=0}^{N-1} \theta_k^2 \norm{\nabla f(x_k) -\nabla f(x_{k+1})}^2_{*} + \frac{\sigma}{2} \sum_{k=0}^{N-1} \norm{\nabla \phi^*(y_k) -\nabla \phi^*(y_{k+1})}^2 \\
   & \quad +\frac{\sigma}{L}\sum_{k=0}^{N-1} \inner{\left(\theta_k^2 - \theta_{k-1}^2 \right)\nabla f(x_k)}{\nabla \phi^*(y_{k+1})} - \frac{\sigma}{L}\sum_{k=0}^{N} \inner{\nabla f(x_k)}{(\theta_k^2 - \theta_{k-1}^2)\nabla \phi^*(y_k)} \\
   & \quad - \frac{\sigma}{L}\sum_{k=1}^{N} \inner{\nabla f(x_k)}{(\theta_{k-1}^2 - \theta_{k-2}^2)\left(\nabla \phi^*(y_{k}) - \nabla \phi^*(y_{k-1}) \right) }.
\end{align*}
Grouping the terms in the last two lines, we obtain
\begin{align*}
    &\frac{\sigma}{L}\sum_{k=0}^{N-1}\left(\theta_k^2 - \theta_{k-1}^2 \right)\big[\inner{\nabla f(x_k)}{\nabla \phi^*(y_{k+1})-\nabla \phi^*(y_k)}-\inner{\nabla f(x_{k+1})}{\nabla \phi^*(y_{k+1})-\nabla \phi^*(y_k)} \big] \\
    = &\; \frac{\sigma}{L} \sum_{k=0}^{N-1}\left(\theta_k^2 - \theta_{k-1}^2 \right) \inner{\nabla f(x_k)-\nabla f(x_{k+1})}{\nabla \phi^*(y_{k+1})-\nabla \phi^*(y_k)}.
\end{align*}
Thus 
\begin{align*}
    \mathbf{U}_{\text{AMD}} =&\; \frac{\sigma}{2L^2}\sum_{k=0}^{N-1} \theta_k^2 \norm{\nabla f(x_k) -\nabla f(x_{k+1})}^2_{*} + \frac{\sigma}{2} \sum_{k=0}^{N-1} \norm{\nabla \phi^*(y_k) -\nabla \phi^*(y_{k+1})}^2\nonumber\\
    & +\frac{\sigma}{L} \sum_{k=0}^{N-1}\left(\theta_k^2 - \theta_{k-1}^2 \right) \inner{\nabla f(x_k)-\nabla f(x_{k+1})}{\nabla \phi^*(y_{k+1})-\nabla \phi^*(y_k)} \nonumber\\
    =& \sum_{k=0}^{N-1} \bigg( \frac{\sigma\theta_k^2  \norm{\nabla f(x_k) -\nabla f(x_{k+1})}^2_{*}}{2L^2} + \frac{\sigma\norm{\nabla \phi^*(y_k) -\nabla \phi^*(y_{k+1})}^2}{2} \nonumber\\
   & \quad + \frac{\sigma\left(\theta_k^2 - \theta_{k-1}^2 \right) }{L}\inner{\nabla f(x_k)-\nabla f(x_{k+1})}{\nabla \phi^*(y_{k+1})-\nabla \phi^*(y_k)}\bigg)\nonumber\\
   \geq &\; 0,
\end{align*}
where we used $\theta_{k}^2 - \theta_{k} \leq \theta_{k-1}^2$ for $k=0,\dots,N-1$ and $\frac{\sigma}{2L^2}\norm{a}_{*}^2 + \frac{\sigma}{2}\norm{b}^2 \geq \frac{\sigma}{L}\norm{a}_{*}\norm{b} \geq \frac{\sigma}{L}\inner{a}{b}$ for any $(a,b)\in X^* \times X$. The first inequality is AM-GM and the second one comes from the definition of a dual norm.\qed
\end{proof}

\paragraph{Comparison with existing methods.}
Our stated \ref{eqn::AMD} and its analysis represent a minor variation of prior work, which are not part of our main contributions. More specifically, when $C=X$ and $\theta_i^2 - \theta_{i-1}^2 = \theta_i$ for $0\leq i\leq N-1$, \ref{eqn::AMD} is equivalent to \cite[Algorithm~22]{d2021acceleration}, except for the last iterate. In the Euclidean setup, \ref{eqn::AMD} is also equivalent to IGA \cite[Section~5]{auslender2006interior}, again excluding the final iterate. 

In terms of the analysis, the prior work \cite{d2021acceleration} achieves the same convergence rate as in Fact~\ref{fact:amd} under the slightly weaker assumption that $f$ is $L$-smooth on $C$, i.e., that $\norm{\nabla f(x)-\nabla f(y)}_{*}\leq L\norm{x-y}$ for all $x,y \in C$. In contrast, we require that $f$ is $L$-smooth on all of $X$, i.e., that $\norm{\nabla f(x)-\nabla f(y)}_{*}\leq L\norm{x-y}$ for all $x,y \in X$. In our proof, we use \eqref{eqn::coco_with_bregman}, which, to the best of our knowledge, requires that $f$ is $L$-smooth on all of $X$, even though our \ref{eqn::AMD} never evaluates $\nabla f$ outside of $C\subset X$. In this sense, our analysis of \ref{eqn::AMD} is slightly less general than the prior analysis of \cite[Algorithm~22]{d2021acceleration}.

Nevertheless, we provide our proof of Fact~\ref{fact:amd} because the structure of our analysis, namely the reliance on \eqref{eqn::coco_with_bregman}, is crucial for the application of our mirror duality result and the construction of our novel method \ref{eqn::AMD-G} in Section~\ref{sec::mirror_duality}.

\paragraph{Lower bound.} 
The classical $\Omega\left(\frac{L\|x_0-x\|^2}{N^2} \right)$ lower bound in the Euclidean setup by \cite{nemirovsky1992information,nesterov2018} applies and establishes optimality of the rate of Fact~\ref{fact:amd} under all choices of  $(f,\phi)$ that satisfy \hyperlink{iiii}{(A1)}.

\section{Mirror Duality and dual-AMD}
\label{sec::mirror_duality}

In Section~\ref{sec::disc_md}, we simply interchanged the roles of $f$ and $\phi^*=\psi^*$ in \ref{eqn::MD}, which has a rate on $f(x_N)$, to obtain a \ref{eqn::dual_MD}, which has a rate on $\psi^*(\nabla f(q_N))$. A natural question that follows is: Can we ``dualize'' \ref{eqn::AMD}, which has an accelerated rate on $f(x_N)$, to obtain a dual method with an accelerated rate on $\psi^*(\nabla f(\cdot))$? Merely interchanging the role of $f$ and $\phi^*=\psi^*$, however, does not work as it produces a method finding an $r_N\in X^*$ reducing $\psi^*(r_N)$, but it does not find a $q_N\in X$ that reduces $\psi^*(\nabla f(q_N))$.

In this section, we obtain a dual method reducing $\psi^*(\nabla f(\cdot))$ at an accelerated $\mathcal{O}(1/N^2)$ rate. The key insight is to interchange the roles of $f$ and $\phi^*=\psi^*$ and also perform an operation analogous to the anti-diagonal transpose as in Section~\ref{sec:prior_works}. We call the resulting operation the \emph{mirror dual}, and we present a duality correspondence as Theorem~\ref{thm::mirror_duality_main}.

\subsection{Coupled first-order method (CFOM) and its mirror dual}
\label{ss:cfom}
Let $X$ be a reflexive Banach space and $f\colon X \rightarrow \RR$ and $\phi^* \colon X^* \rightarrow \RR$ be differentiable.
An \emph{$N$-step fixed-step coupled first-order method} (CFOM) with step sizes $\{a_{k,i}\}_{0\leq i<k\leq N}$ and $\{b_{k,i}\}_{0\leq i\leq k\leq N}$ is:

\begin{align}  
\begin{split}
    y_{k+1} = y_k - \sum\limits_{i=0}^{k} a_{k+1,i}\nabla f(x_i) , \quad x_{k+1} = x_k - \sum\limits_{i=0}^{k+1} b_{k+1,i}\nabla \phi^{*}(y_{i})\quad
\end{split}\label{eqn::CFSFOM}
\end{align}
for $k=0,\dots, N-1$, where $y_0 \in X^*$ is a starting point and $x_0=\nabla \phi^*(y_0)$.
Note that the upper limits of the two summations are $k$ and $k+1$, which means $y_{k+1}$ is computed before $x_{k+1}$. Although $b_{0,0}$ is not used in \eqref{eqn::CFSFOM}, define $b_{0,0}=-1$ for the sake of notational simplicity in later arguments. The class of CFOMs is very general, and it, in particular, includes \ref{eqn::MD} and \ref{eqn::AMD}.

Let $\psi^* \colon X^* \rightarrow \RR$ be differentiable. Throughout this section, we assume $\psi^*(0)=0$ and $0\in X ^*$ is the unique minimizer of $\psi^*$.
For a given CFOM \eqref{eqn::CFSFOM}, we define its \emph{mirror dual} as
\begin{align}
\begin{split} \label{eqn::dual_alg_def}
     q_{k+1} = q_k - \sum_{i=0}^{k} a_{N-i,N-1-k}\nabla \psi^{*}(r_i), \quad  r_{k+1} = r_k - \sum_{i=0}^{k+1} b_{N-i,N-1-k}\nabla f(q_i)
\end{split} 
\end{align}
for $k=0,1,\dots,N-1$, where $q_0 \in X $ is a starting point and $r_0 = -b_{N,N}\nabla f(q_0)$.
The mirror dual has the coefficients $\{a_{k,i}\}_{0\leq i<k\leq N}$ and $\{b_{k,i}\}_{0\leq i\leq k\leq N}$ used in an anti-diagonal transposed manner as in the H-dual and, furthermore, has the roles of $f$ and $\phi^*$ swapped when $\psi=\phi$. 
In Section~\ref{ss:h-duality-relation}, we more precisely discuss how the mirror dual and mirror duality generalize the prior notions of the H-dual and H-duality.

\begin{proposition} \label{thm::mirror_duality_main_2}
For a CFOM \eqref{eqn::CFSFOM}, assume $1 - \sum_{i=0}^{k}\sum_{j=0}^{i+1} b_{i+1,j} =1$ for $k=0,\dots,N-1$, which is a sufficient condition for $x_k\in \mathrm{conv}\{\nabla \phi^{*}(y_0),\dots,\nabla \phi^{*}(y_k)\}$ for $k=0,\dots , N$. Then, the mirror dual \eqref{eqn::dual_alg_def} satisfies $r_N =\nabla f(q_N)$.
\end{proposition}

Under the condition of Proposition~\ref{thm::mirror_duality_main_2}, a convergence guarantee on $\psi^*(r_N)$ can be translated to a bound on $\psi^*(r_N)=\psi^*(\nabla f(q_N))$. As we showed in the proof of Fact~\ref{fact:amd}, the CFOM \ref{eqn::AMD} satisfies the property $x_k\in \mathrm{conv}\{\nabla \phi^*(y_0),\dots,\nabla \phi^*(y_k)\}$ for $k=0,\dots , N$. Therefore, Proposition~\ref{thm::mirror_duality_main_2} applies to the mirror dual of \ref{eqn::AMD} that we present in Section~\ref{ss:dual-amd}.

\begin{proof}
Observe first that
\begin{align*}
    x_{k+1} = x_0 - \sum_{i=0}^{k} (x_i -x_{i+1}) = \nabla \phi^*(y_0) - \sum_{i=0}^{k}\sum_{j=0}^{i+1} b_{i+1,j}\nabla \phi^*(y_j)
\end{align*}
for $k=0,\dots,N-1$. Therefore, if $1 - \sum_{i=0}^{k}\sum_{j=0}^{i+1} b_{i+1,j} =1$ for $k=0,\dots,N-1$, $x_{k+1} \in \mathrm{conv}\{\nabla \phi^*(y_0),\dots, \nabla \phi^*(y_{k+1})\}$. Now we prove the mirror dual satisfies $r_N=\nabla f(q_N)$. Note that
\begin{align} \label{eqn::appendix_convex_sum_pf_1}
    \sum_{i=0}^{k+1} b_{k+1,i} =0,\quad k=0,\dots,N-1.
\end{align}
Thus, we get that the following holds for the mirror dual:
\begin{align*}
    r_N = & r_0 - \sum_{i=0}^{N-1}\sum_{j=0}^{i+1} b_{N-j,N-1-i}\nabla f(q_j) \\
    =& -b_{N,N}\nabla f(q_0) -  \sum_{i=0}^{N-1} b_{N,N-1-i} \nabla f(q_0) - \sum_{j=1}^{N}\left(\sum_{i=j-1}^{N-1} b_{N-j,N-1-i} \right) \nabla f(q_j)\qquad \\
    =& -b_{N,N}\nabla f(q_0) -  \sum_{i=0}^{N-1} b_{N,N-1-i} \nabla f(q_0) - \sum_{j=1}^{N}\left(\sum_{l=0}^{N-j} b_{N-j,l} \right) \nabla f(q_j) \\
    =& -b_{N,N}\nabla f(q_0) -  \sum_{i=0}^{N-1} b_{N,N-1-i} \nabla f(q_0) - b_{0,0}\nabla f(q_N) \qquad \left( \because \eqref{eqn::appendix_convex_sum_pf_1} \right) \\
    =& -\sum_{i=0}^{N} b_{N,N-i}\nabla f(q_0) + \nabla f(q_N) \qquad \left(\because b_{0,0}=-1 \right)\\
    =& \nabla f(q_N).\qquad \left( \because \eqref{eqn::appendix_convex_sum_pf_1} \right)
\end{align*}    
\qed\end{proof}

\subsection{Mirror Duality} \label{ss::mirror_duality}
% {\color{red}guarantee on $x$ or $x^\star$?}
For an $N$-step CFOM \eqref{eqn::CFSFOM}, which we denote as $\cA$, consider establishing a convergence guarantee on $f(x_N)-f(x)$ for $x\in \dom \phi$ using the following proof strategy. 
 Assume \hyperlink{iiii}{(A1)} for $(f,\phi)$, i.e., assume $f$ is $L$-smooth and $\phi$ is $\sigma$-strongly convex.
For a positive nondecreasing sequence $\{u_i\}_{i=0}^{N}\subset \RR$, define $\{\mathcal{U}_k\}_{k=0}^{N}$ as 
\begin{align*}
  \cU_0 = \phi(x)+\phi^*(y_0)-\inner{y_0}{x} - u_0D_f(x,x_0) 
\end{align*}
and
\begin{align*}
\quad \mathcal{U}_{k+1} = &\; \mathcal{U}_k - (u_{k+1}-u_k)D_f\left(x , x_{k+1}\right)\\
&- u_{k}\Big(
\underbrace{D_f\left(x_k,x_{k+1}\right)-\frac{1}{2L}\norm{\nabla f(x_k)-\nabla f(x_{k+1})}_{*}^2 }_{\ge 0\text{ by \eqref{eqn::coco_with_bregman}}}\Big) \\
 & - \Big(
 \underbrace{D_{\phi^*}\left(y_k , y_{k+1} \right)-\frac{\sigma}{2}\norm{\nabla \phi^{*}(y_{k+1})-\nabla \phi^*(y_{k})}^2 }_{\ge 0\text{ by \eqref{eqn::coco_with_bregman}}}\Big)
\end{align*}
for $k=0,\dots,N-1$. 
The sequence $\{\cU_k\}_{k=0}^{N}$ is nonincreasing by construction.
Define $\mathbf{U}_{\mathcal{A}}$ via
\begin{align*}
    \mathcal{U}_N =&\; u_N\left(f(x_N)-f(x) \right) + 
    \underbrace{\phi(x)+\phi^*(y_N)-\inner{y_N}{x}}_{\ge 0 \text{ by \eqref{eqn::fenchel}}}\\
    & \quad +\inner{y_N-y_0+\sum_{i=0}^{N}(u_i-u_{i-1})\nabla f(x_i)}{x }+ \mathbf{U}_{\cA}.
\end{align*}
If $y_N-y_0+\sum_{i=0}^{N}(u_i-u_{i-1})\nabla f(x_i)=0$  and $ \mathbf{U}_\cA \geq 0$, then we conclude
\begin{align*}
 u_N( f(x_N)-f(x)) \leq  \cU_N\leq \dots  \leq \cU_{0} \leq \phi(x)+\phi^*(y_0)-\inner{y_0}{x}.
\end{align*}
By a telescoping sum argument, the $f$- and $\phi^*$-function values cancel out and $\mathbf{U}_\cA$ is a function of only $\{\nabla f(x_i)\}_{i=0}^{N} \subset X^*$ and $\{\nabla \phi^*(y_i)\}_{i=0}^{N}\subset X$. (We state the exact form of $\mathbf{U}_\cA$ later in \eqref{eqn::appendix_U_first}.) If we show $ \inf \mathbf{U}_\cA\ge 0$, where the infimum is taken over all possible values of $\{\nabla f(x_i)\}_{i=0}^{N}$ and $\{\nabla \phi^*(y_i)\}_{i=0}^{N}$, then we conclude $ \mathbf{U}_\cA \geq 0$ and $ u_N( f(x_N)-f(x))  \leq \phi(x)+\phi^*(y_0)-\inner{y_0}{x}$. Recall that this is the exact strategy we employed in our proof of Fact~\ref{fact:amd}. With a slight abuse of notation, define $\mathbf{U}_{\cA}\colon (X^*)^{N+1}\times (X)^{N+1}\rightarrow \mathbb{R}$ so that
\begin{equation}
\mathbf{U}_{\cA}=
\mathbf{U}_{\cA}(\nabla f(x_0),\dots,\nabla f(x_N),\nabla \phi^*(y_0),\dots,\nabla \phi^*(y_N)).
\label{eq:UA_functional_form}
\end{equation}

On the other hand, consider establishing a convergence rate of $\psi^*(r_N)$ for an $N$-step dual CFOM \eqref{eqn::dual_alg_def}, which we denote as $\cB$. Assume \hyperlink{iiii}{(A1)} for $(f,\psi)$, i.e., $f$ is $L$-smooth and $\psi$ is $\sigma$-strongly convex. For a positive increasing sequence $\{v_i\}_{i=0}^{N}\subset \RR$, define $\{\cV_k\}_{k=0}^{N}$ as 
\[
\cV_0 = v_0\left(f(q_0)-f(q_N) \right)
\]
and
\begin{align*}
    \cV_{k+1} = & \;\cV_k - (v_{k+1}-v_k)D_f\left(q_N, q_k\right)\\
    &\quad- v_{k+1}\Big(\underbrace{D_f(q_k,q_{k+1})-\frac{1}{2L}\norm{\nabla f(q_k) -\nabla f(q_{k+1})}^2_{*}}_{\ge 0\text{ by \eqref{eqn::coco_with_bregman}}} \Big) \\
    &\quad - \Big(\underbrace{D_{\psi^*}\left(r_k,r_{k+1} \right) - \frac{\sigma}{2}\norm{\nabla \psi^*(r_{k+1})-\nabla \psi^*(r_k)}^2}_{\ge 0\text{ by \eqref{eqn::coco_with_bregman}}} \Big)
\end{align*}
for $k=0,1,\dots,N-1$. The sequence $\{\cV_k\}_{k=0}^{N}$ is nonincreasing by construction.
Define $\mathbf{V}_{\cB}$ via
\begin{align*}
        \mathcal{V}_N = \psi^*(r_N) + \underbrace{D_{\psi^*}\left(0,r_0\right)}_{\ge 0}+\mathbf{V}_{\cB}.
\end{align*}
Here, we used $\psi^*(0)=0$. If $ \mathbf{V}_\cB \geq 0$, then we conclude
\begin{align*}
  \psi^{*}(r_N) \leq  \cV_N \leq\dots \leq \cV_0 = v_0(f(q_0)-f(q_N)) .
\end{align*}
Again, by a telescoping sum argument, the $f$- and $\psi^*$-function values cancel out and $\mathbf{V}_\cB$ is a function of only $\{\nabla f(q_i)\}_{i=0}^{N} \subset X^*$ and $\{\nabla \psi^*(r_i)\}_{i=0}^{N}\subset X$. (We state the exact form of $\mathbf{V}_\cB$ later in \eqref{eqn::appendix_V_first}.) If we show $ \inf \mathbf{V}_\cB\ge 0$, where the infimum is taken over all possible values of $\{\nabla f(x_i)\}_{i=0}^{N}$ and $\{\nabla \psi^*(y_i)\}_{i=0}^{N}$, then we conclude $ \mathbf{V}_\cB \geq 0$ and
\begin{align*}
\psi^{*}(r_N) \leq v_0(f(q_0)-f(q_N))\leq v_0(f(q_0)-\inf_{x\in X} f(x)).
\end{align*}
With a slight abuse of notation, define $\mathbf{V}_{\cB}\colon (X^*)^{N+1}\times (X)^{N+1}\rightarrow \mathbb{R}$ so that
\begin{equation}
\mathbf{V}_{\cB}=
\mathbf{V}_{\cB}(\nabla f(x_0),\dots,\nabla f(x_N),\nabla \psi^*(y_0),\dots,\nabla \psi^*(y_N)).
\label{eq:VB_functional_form}
\end{equation}
The following result establishes a one-to-one correspondence between these two proof strategies.
\begin{theorem}[Mirror Duality] \label{thm::mirror_duality_main}
Let $X$ be a reflexive Banach space and $f\colon X \rightarrow \RR$ and $\phi\colon X \rightarrow \RR$ be CCP. Assume \hyperlink{iiii}{(A1)}. Let $\cA$ be a CFOM \eqref{eqn::CFSFOM} and $\cB$ be its mirror dual CFOM \eqref{eqn::dual_alg_def}.
Let $\{u_i\}_{i=0}^{N}\in \RR$ be a positive nondecreasing sequence and $v_i = \frac{1}{u_{N-i}}$ for $i=0,\dots, N$. Let 
\begin{align*}
&\mathbf{U}_{\cA}\colon (X^*)^{N+1}\times (X)^{N+1}\rightarrow \mathbb{R}\\
&\mathbf{V}_{\cB}\colon (X^*)^{N+1}\times (X)^{N+1}\rightarrow \mathbb{R}
\end{align*} respectively be defined as in \eqref{eq:UA_functional_form} and \eqref{eq:VB_functional_form}.
Then,
\begin{align*}
    &\inf\{\mathbf{U}_\cA(x_0^*,\dots,x_{N}^*,x_0,\dots,x_{N})\,|\,x_0^*,\dots,x_{N}^*\in X^*,\,x_0,\dots,x_{N}\in X^*\}\\
    =  &\inf 
    \{\mathbf{V}_{\cB}(x_0^*,\dots,x_{N}^*,x_0,\dots,x_{N})\,|\,x_0^*,\dots,x_{N}^*\in X^*,\,x_0,\dots,x_{N}\in X^*\}.
\end{align*}
\end{theorem}

Proposition~\ref{thm::mirror_duality_main_2} and Theorem~\ref{thm::mirror_duality_main} together provide a sufficient condition that ensures a CFOM $\cA$ with a convergence guarantee on $f(x_N)-f(x)$ can be ``mirror-dualized'' to obtain a mirror dual CFOM $\cB$ with a convergence guarantee on $\psi^*(r_N)=\psi^*(\nabla f(q_N))$, and vice versa.

\begin{proof}
First, we calculate $\mathbf{U}_{\cA}$. Denote $\nabla f(x_{N+1})=0$ for notational convenience.
\begin{align}
\begin{split} \label{eqn::appendix_U_first}
    \mathbf{U}_{\cA} = &\sum_{k=0}^{N} (u_{k}-u_{k-1})\inner{\nabla f(x_{k})}{-x_{k}}\\
    &+ \sum_{k=0}^{N-1} u_k\left(\inner{\nabla f(x_{k+1})}{x_k-x_{k+1}}+\frac{1}{2L}\norm{\nabla f(x_k)-\nabla f(x_{k+1})}^2_{*} \right)  \\
    &\quad + \sum_{k=0}^{N-1} \left(\inner{y_k-y_{k+1}}{\nabla \phi^*(y_{k+1})}+\frac{\sigma}{2}\norm{\nabla \phi^*(y_{k+1})-\nabla \phi^*(y_k)}^2 \right) \\
    =& \sum_{k=0}^{N-1} \frac{u_{k}}{2L}\norm{\nabla f(x_k)-\nabla f(x_{k+1})}^2_{*} + \sum_{k=0}^{N-1}\frac{\sigma}{2}\norm{\nabla \phi^*(y_{k})-\nabla \phi^*(y_{k+1})}^2 \\
    & \quad +\sum_{k=0}^{N-1} \inner{y_{k}-y_{k+1}}{\nabla \phi^*(y_{k+1})} - \sum_{k=0}^{N} u_k\inner{\nabla f(x_k)-\nabla f(x_{k+1})}{x_k}. 
\end{split}
\end{align} 
By the definition of CFOM \eqref{eqn::CFSFOM}, $y_{k-1}-y_k \in \mathrm{span}\{\nabla f(x_0),\dots, \nabla f(x_{k-1})\}$ and $x_k \in x_0 + \mathrm{span}\{\nabla \phi^*(y_0),\dots, \nabla \phi^*(y_k)\}$. Therefore, $\mathbf{U}_{\cA}\colon  \left(X^{*}\right)^{N+1}\times X^{N+1}  \rightarrow \RR$ is a function of $\{a_{k,i}\}_{0\leq i<k\leq N}$ and $\{b_{k,i}\}_{0\leq i\leq k\leq N}$. Next we calculate $\mathbf{V}_{\cB}$.
\begin{align}
\begin{split} \label{eqn::appendix_V_first}
\mathbf{V}_\cB = &\;\cV_N - \psi^*(r_N) - D_{\psi^*}(0,r_0)-\psi^*(r_N) + \psi^*(0) \\
=& \sum_{k=0}^{N-1} (v_{k+1}-v_k)\inner{\nabla f(q_k)}{q_N-q_k}\\
&+ \sum_{k=0}^{N-1} v_{k+1}\left(\inner{\nabla f(q_{k+1})}{q_k-q_{k+1}}+\frac{1}{2L}\norm{\nabla f(q_k)-\nabla f(q_{k+1})}^2_{*} \right) \\
&  + \sum_{k=0}^{N-1}\left( \inner{r_k-r_{k+1}}{\nabla \psi^*(r_{k+1})} + \frac{\sigma}{2}\norm{\nabla \psi^*(r_{k})-\nabla \psi^*(r_{k+1})}^2\right)+\inner{\nabla \psi^*(r_0)}{-r_0} \\
= & \sum_{k=0}^{N-1} \frac{v_{k+1}}{2L} \norm{\nabla f(q_k)-\nabla f(q_{k+1})}^2_{*} + \sum_{k=0}^{N-1} \frac{\sigma}{2}\norm{\nabla \psi^*(r_{k})-\nabla \psi^*(r_{k+1})}^2 \\
&  + \sum_{k=0}^{N} \inner{r_{k-1}-r_{k}}{\nabla \psi^*(r_{k})} \\&  +\sum_{k=0}^{N-1} \inner{v_{k+1}\nabla f(q_{k+1}) - (v_{k+1}-v_k)\nabla f(q_k) -\dots - (v_1 - v_0)\nabla f(q_0)}{q_k-q_{k+1}}
\end{split}
\end{align}
where we have used $r_{-1}=0$. By the definition of mirror dual, $r_{k-1}-r_k \in \mathrm{span}\{\nabla f(q_0),\dots,\nabla f(q_k) \}$ and $q_k - q_{k+1} \in \mathrm{span} \{\nabla \psi^*(r_0),\dots, \nabla \psi^*(r_k) \}$. Therefore, $\mathbf{V}_{\cB} \colon  \left(X^{*}\right)^{N+1} \times X^{N+1} \rightarrow \RR$ is the function of $\{a_{k,i}\}_{0\leq i<k\leq N}$ and $\{b_{k,i}\}_{0\leq i\leq k\leq N}$. To show $\inf \mathbf{U}_{\cA}= \inf \mathbf{V}_{\cB}$ under $v_i = \frac{1}{u_{N-i}}$ for $i=0,\dots,N$, we show the following statement: $\forall \,\left(\{A_i\}_{i=0}^{N}, \{B_i\}_{i=0}^{N}   \right)\in (X^*)^{N+1} \times \left(X\right)^{N+1} ,$
\begin{align*}
     \mathbf{U}_\cA \left(\{A_i\}_{i=0}^{N}, \{B_i\}_{i=0}^{N}   \right) = \mathbf{V}_{\cB}\left( \{C_i\}_{i=0}^{N},\{D_i\}_{i=0}^{N}\right),
\end{align*}
where $\{C_i\}_{i=0}^{N}$ and $\{D_i\}_{i=0}^{N}$ are defined as
\begin{align} \label{eqn::appendix_mirror_transform}
\begin{split}
     C_0\colon =u_N A_N ,\quad & C_{N-i} - C_{N-i-1}\colon = u_i\left(A_i - A_{i+1} \right)  , \quad i=0,1,\dots, N-1 \\
    & D_i \colon = B_{N-i},\quad i=0,1,\dots, N.
\end{split}
\end{align}
Note that the above transformation is a bijection since $u_i$'s are positive, thus if we can show $\mathbf{U}_\cA \left(\{A_i\}_{i=0}^{N}, \{B_i\}_{i=0}^{N}   \right) = \mathbf{V}_{\cB}\left( \{C_i\}_{i=0}^{N},\{D_i\}_{i=0}^{N}\right)$, $\inf \mathbf{U}_{\cA}= \inf \mathbf{V}_{\cB}$. First we write $\mathbf{U}_\cA$ and $\mathbf{V}_\cB$ with $\left(\{A_i\}_{i=0}^{N}, \{B_i\}_{i=0}^{N}   , \{C_i\}_{i=0}^{N},\{D_i\}_{i=0}^{N}\right)$.
\begin{align*}
    \mathbf{U}_{\cA}=& \sum_{k=0}^{N-1} \frac{u_{k}}{2L}\norm{A_k-A_{k+1}}^2_{*} + \sum_{k=0}^{N-1}\frac{\sigma}{2}\norm{B_k - B_{k+1}}^2 \\
    & \quad +\sum_{k=0}^{N-1} \inner{\sum_{i=0}^{k}a_{k+1,i}A_i}{B_{k+1}} - \sum_{k=0}^{N} u_k\inner{A_k - A_{k+1}}{x_k}
\end{align*}
where $x_{k+1} \colon= x_k -\sum_{i=0}^{k+1} b_{k+1,i} B_i$ and $x_0 = B_0$. For simplicity, define $A_{k+1}=C_{-1}= 0$. 
Using \eqref{eqn::appendix_mirror_transform} gives us
\begin{align*}
    \sum_{k=0}^{N-1} \inner{\sum_{i=0}^{k}a_{k+1,i}A_i}{B_{k+1}} = \sum_{k=0}^{N-1}\sum_{i=0}^{k} a_{k+1,i}\inner{A_i}{D_{N-k-1}}
\end{align*}
and
\begin{align*}
    - \sum_{k=0}^{N} u_k\inner{A_k - A_{k+1}}{x_k} = & - \sum_{k=0}^{N} \inner{C_{N-k}-C_{N-k-1}}{x_k} \\
    = & \inner{C_N}{-x_0} + \inner{C_{N-1}}{x_0-x_1} + \dots + \inner{C_0}{x_{N-1}-x_N} \\
    = & \sum_{k=-1}^{N-1} \inner{C_{N-k-1}}{\sum_{i=0}^{k+1} b_{k+1,i}B_i} \\
    = & \sum_{k=0}^{N}\sum_{i=0}^k b_{k,i} \inner{C_{N-k}}{B_i}.
\end{align*}
Combining the above equations, we conclude
\begin{align} \label{eqn::appendix_U_final}
\begin{split}
    \mathbf{U}_{\cA}=& \sum_{k=0}^{N-1} \frac{u_{k}}{2L}\norm{A_k-A_{k+1}}^2_{*} + \sum_{k=0}^{N-1}\frac{\sigma}{2}\norm{B_k - B_{k+1}}^2 \\
    & \quad +\sum_{k=0}^{N-1}\sum_{i=0}^{k} a_{k+1,i}\inner{A_i}{D_{N-k-1}} + \sum_{k=0}^{N}\sum_{i=0}^k b_{k,i} \inner{C_{N-k}}{B_i}.
\end{split}
\end{align}
In the same way, $\mathbf{V}_\cB$ is written as follows. 
\begin{align*}
\mathbf{V}_\cB = & \sum_{k=0}^{N-1} \frac{v_{k+1}}{2L} \norm{C_k-C_{k+1}}^2_{*} + \sum_{k=0}^{N-1} \frac{\sigma}{2}\norm{D_k-D_{k+1}}^2 \\
& \quad + \sum_{k=0}^{N} \inner{\sum_{i=0}^{k} b_{N-i,N-k}C_i} {D_k}\\& \quad +\sum_{k=0}^{N-1} \inner{v_{k+1}C_{k+1} - (v_{k+1}-v_k)C_k -\dots - (v_1 - v_0)C_0}{\sum_{i=0}^{k} a_{N-i,N-1-k}D_i}
\end{align*}
By \eqref{eqn::appendix_mirror_transform},
\begin{align*}
    &\sum_{k=0}^{N} \inner{\sum_{i=0}^{k} b_{N-i,N-k}C_i}{D_k} = \sum_{k=0}^{N}\sum_{i=0}^{k} b_{N-i,N-k} \inner{C_i}{D_k}  = \sum_{k=0}^{N}\sum_{i=0}^{k} b_{N-i,N-k} \inner{C_i}{B_{N-k}}, \\
    &\sum_{k=0}^{N-1} \inner{v_{k+1}C_{k+1} - (v_{k+1}-v_k)C_k -\dots - (v_1 - v_0)C_0}{\sum_{i=0}^{k} a_{N-i,N-1-k}D_i} \\
    \stackrel{(\star)}{=}&  \sum_{k=0}^{N-1} \inner{\frac{1}{u_{N-k-1}}\left(C_{k+1}-C_k\right) + \frac{1}{u_{N-k}}\left(C_k - C_{k-1} \right)+\dots  + \frac{1}{u_N}C_0 }{\sum_{i=0}^{k} a_{N-i,N-1-k}D_i} \\
    =& \sum_{k=0}^{N-1} \inner{(A_{N-k-1}-A_{N-k})+ (A_{N-k}-A_{N-k+1})+\dots + A_N}{\sum_{i=0}^{k} a_{N-i,N-1-k}D_i} \\
    =& \sum_{k=0}^{N-1} \inner{A_{N-k-1}}{\sum_{i=0}^{k} a_{N-i,N-1-k}D_i} \\
    =& \sum_{k=0}^{N-1}\sum_{i=0}^{k} a_{N-i,N-k-1}\inner{A_{N-k-1}}{D_i}.
\end{align*}
$v_i = \frac{1}{u_{N-i}}$ is used at $(\star)$.
\begin{align}
\begin{split} \label{eqn::appendix_V_final}
    \mathbf{V}_{\cB} =  & \sum_{k=0}^{N-1} \frac{v_{k+1}}{2L} \norm{C_k-C_{k+1}}^2_{*} + \sum_{k=0}^{N-1} \frac{\sigma}{2}\norm{D_k-D_{k+1}}^2 \\
& \quad +\sum_{k=0}^{N}\sum_{i=0}^{k} b_{N-i,N-k} \inner{C_i}{B_{N-k}} +\sum_{k=0}^{N-1}\sum_{i=0}^{k} a_{N-i,N-k-1}\inner{A_{N-k-1}}{D_i} \\
\stackrel{(\circ)}{=} & \sum_{k=0}^{N-1} \frac{u_i}{2L}\norm{A_k-A_{k+1}}_{*}^2 + \sum_{k=0}^{N-1} \frac{\sigma}{2}\norm{B_k-B_{k+1}}^2 \\
& \quad +\sum_{k=0}^{N}\sum_{i=0}^{k} b_{N-i,N-k} \inner{C_i}{B_{N-k}} +\sum_{k=0}^{N-1}\sum_{i=0}^{k} a_{N-i,N-k-1}\inner{A_{N-k-1}}{D_i} \\
\stackrel{(\bullet)}{=} & \sum_{k=0}^{N-1} \frac{u_i}{2L}\norm{A_k-A_{k+1}}_{*}^2 + \sum_{k=0}^{N-1} \frac{\sigma}{2}\norm{B_k-B_{k+1}}^2 \\
& \quad + \sum_{l=0}^{N}\sum_{j=0}^{l} b_{l,j} \inner{C_{N-l}}{B_j} + \sum_{l=0}^{N-1}\sum_{j=0}^l a_{l+1,j}\inner{A_j}{D_{N-l-1}} \\
\stackrel{(\bullet)}{=} & \sum_{k=0}^{N-1} \frac{u_i}{2L}\norm{A_k-A_{k+1}}_{*}^2 + \sum_{k=0}^{N-1} \frac{\sigma}{2}\norm{B_k-B_{k+1}}^2 \\
& \quad + \sum_{k=0}^{N}\sum_{i=0}^{k} b_{k,i} \inner{C_{N-k}}{B_i} + \sum_{k=0}^{N-1}\sum_{i=0}^k a_{k+1,i}\inner{A_i}{D_{N-k-1}}
\end{split}
\end{align}
\eqref{eqn::appendix_mirror_transform} is used at $(\circ)$. $(\bullet)$s are index changing. Now it becomes clear that \eqref{eqn::appendix_V_final} is the same as \eqref{eqn::appendix_U_final}.
\qed\end{proof}

\subsection{dual-AMD}
\label{ss:dual-amd}

Again, consider the problem
\begin{align*}
\begin{array}{ll}
\underset{x \in X}{\mbox{find}}&\quad 0=\nabla f(x),
\end{array}
\end{align*} 
where $X$ is a reflexive Banach space and $f\colon X\to\RR$ is a differentiable CCP function. Given an approximate solution, we measure its suboptimality via
\[
\psi^*(\nabla f(x)).
\]
Recall that we assumed $\psi^*(0)=0$ and $0\in X^*$ is the unique minimizer of $\psi^*$. Consider \hyperlink{iiii}{(A1)}, which we restate as
\begin{itemize}
     \item [\hypertarget{iiiii}{(\hyperlink{iiiii}{A1})}]
     $f$ is $L$-smooth; $\psi$ is $\sigma$-strongly convex; $L>0$, $\sigma>0$.
\end{itemize}
We propose dual accelerated mirror descent \eqref{eqn::AMD-G}
\begin{align}
                & q_{k+1} = q_k - \frac{\sigma}{L}(\theta_{N-k-1}^2 - \theta_{N-k-2}^2) \nabla \psi^* (r_k)\nonumber\\
                & g_{k+1}=g_k + \frac{1}{\theta_{N-k-1}^2}\left(\nabla f(q_{k+1}) -\nabla f(q_k) \right)
        \label{eqn::AMD-G}
        \tag{dual-AMD}\\
                & r_{k+1} = r_k + (\theta_{N-k-1}^2 - \theta_{N-k-2}^2)\left(g_{k+1}-g_k \right)+(\theta_{N-k-2}^2 - \theta_{N-k-3}^2)g_{k+1}
\nonumber
\end{align}
\begin{lemma}\label{lemma::amd_and_dual_amd}
\ref{eqn::AMD-G} is the mirror dual of \ref{eqn::AMD}.
\end{lemma}
We defer the proof to Appendix~\ref{Appendix:Banach}, which follows from direct calculations of the coefficients. Now, our Mirror Duality theorem and Lemma~\ref{lemma::amd_and_dual_amd} immediately imply a convergence guarantee on \ref{eqn::AMD-G}.
\begin{corollary}[dual-AMD]\label{thm::AMD-G}
Let $X$ be a reflexive Banach space and $f\colon X\to\RR$ and $\psi\colon X\to\RR$ be CCP. Assume \hyperlink{iiiii}{(A1)}.
Then, the iterates of \ref{eqn::AMD-G} are well defined and $r_N = \nabla f(q_N)$. If we further assume that $\psi^*(0)=0$ and $0\in X^*$ is the unique minimizer of $\psi^*$, then
\begin{align*}
       \psi^{*}(\nabla f(q_N)) = \psi^{*}(r_N)  \leq \frac{L}{\sigma \theta_N^2}\Big(f(q_0)-\inf_{x\in X}f(x) \Big) \overset{(\bullet)}{=} \cO\bigg(\frac{L(f(q_0)-\inf_{x\in X}f(x)) }{\sigma N^2}\bigg),
\end{align*}
where  $(\bullet)$ holds if $\theta_{i}^2-\theta_{i}= \theta_{i-1}^2$ for $0 \leq i\leq N-1$.
 (If $\inf_{x\in X}f(x)=-\infty$, then the right-hand-side is $+\infty$ and the bound vacuously holds.)
\end{corollary}
\begin{proof}
As shown in Fact~\ref{fact:amd}, the convergence rate of \ref{eqn::AMD} is proved by $\{\cU_k\}_{k=0}^{N}$ with $u_i=\frac{\sigma}{L}\theta_i^2$, $\inf \mathbf{U}_{\text{AMD}} \geq 0$, and $x=x_\star$. Since $u_i=\frac{\sigma}{L}\theta_i^2$ are positive and non-decreasing, $v_i=  \frac{1}{u_{N-i}}$ are also positive and non-decreasing.

Now, we use Theorem~\ref{thm::mirror_duality_main} and Proposition~\ref{thm::mirror_duality_main_2}. Since \ref{eqn::AMD-G} is the mirror dual of \ref{eqn::AMD} (Lemma~\ref{lemma::amd_and_dual_amd}),Theorem~\ref{thm::mirror_duality_main} provides $\inf \mathbf{V}_{\text{dual-AMD}}=\inf \mathbf{U}_{\text{AMD}}\geq 0.$ Therefore, $\psi^*(r_N)\leq \frac{L}{\sigma}\frac{1}{\theta_N^2}\left(f(q_0)-f_\star \right)$ holds for \ref{eqn::AMD-G} and Proposition~\ref{thm::mirror_duality_main_2} gives us $r_N = \nabla f(q_N)$, which leads to
\begin{align*}
    \psi^{*}(\nabla f(q_N))  = \psi^{*}(r_N)  \leq \frac{L}{\sigma\theta_N^2}\left(f(q_0)-f_\star \right).
\end{align*}
\qed\end{proof}

\subsection{Concatenation of \ref{eqn::AMD} and \ref{eqn::AMD-G}}
\label{sec::concat}
Concatenating \ref{eqn::AMD} and \ref{eqn::AMD-G}, we obtain a $\cO\left(1/N^4 \right)$ rate on $\psi^*(\nabla f(\cdot))$.
More specifically, assume $C=X$, so that $f_\star=\inf_{x\in C}f(x)=\inf_{x\in X}f(x)$.  Also assume $f$ is $L$-smooth convex, $\phi$ is $\sigma_1$-strongly convex, and $\psi$ is $\sigma_2$-strongly convex. 
Starting from $x_0=\nabla \phi^*(y_0)$, execute $N$ iterations of \ref{eqn::AMD} to obtain $X_N$. Then, starting from $q_0=x_N$, execute $N$ iterations of \ref{eqn::AMD-G} and denote the output as $x_{2N}=q_N$.
\begin{corollary}
\label{cor:amd+amdg}
Assume \hyperlink{iiiii}{(A1)} and $x_\star \in \dom \phi \cap \mathrm{argmin}f$ exists. Then concatenation of \ref{eqn::AMD} and \ref{eqn::AMD-G} has the rate
\begin{align*}
    &\psi^*(\nabla f(x_{2N}))\underbrace{\leq}_{\text{\ref{eqn::AMD-G}}} \frac{L}{\sigma_1 \theta_N^2}\left(f(x_N)-f(x_\star) \right) \\&\qquad\underbrace{\leq}_{\text{\ref{eqn::AMD}}} \frac{L^2 D_\phi\left(x_\star,x_0\right)}{\sigma_1 \sigma_2 \theta_N^4} \stackrel{(\bullet)}{=}\mathcal{O}\left(\frac{L^2D_\phi\left(x_\star,x_0\right)}{\sigma_1 \sigma_2 N^4} \right),
\end{align*}
where  $(\bullet)$ holds if $\theta_{i}^2-\theta_{i}= \theta_{i-1}^2$ for $0 \leq i\leq N-1$.
\end{corollary}

\paragraph{Lower bound.} 
Under \hyperlink{iiiii}{(A1)}, the classical $\Omega\Big(\frac{L^2 \|x_0-x_\star\|^2}{N^4}\Big)$ lower bound by \cite{nemirovsky1991optimality,nemirovsky1992information}
applies and establishes optimality of this concatenated rate. This also implies optimality of \ref{eqn::AMD-G} under all choices of $(f,\phi)$ that satisfies \hyperlink{iiii}{(A1)}, since if the rate of \ref{eqn::AMD-G} could be improved, then the concatenated rate of \ref{eqn::AMD} and \ref{eqn::AMD-G} would be improved and would violate the aforementioned lower bound.

\section{Applications}
\label{s:applications}
In this section, we provide two applications of \ref{eqn::AMD-G} and \ref{eqn::AMD}.
Our discussion of these applications very closely follows \cite{diakonikolas2023complementary}.

\subsection{Small gradients with respect to $\norm{\cdot}_q$-norm}
\label{sec::l_p_norm_grad}

For $p\in (1,2]$ and $q=\frac{p}{p-1}$, let $f\colon \mathbb{R}^n\rightarrow\mathbb{R}$ be a differentiable convex function that is $L$-smooth with respect to $\norm{\cdot}_p$. Our goal is to find an $x\in \mathbb{R}^n$ such that $\|\nabla f(x)\|_q$ is small, and the concatenation of \ref{eqn::AMD} and \ref{eqn::AMD-G} yields an optimal method for this task, improving upon the prior state-of-the-art rate of \cite{diakonikolas2023complementary}.

For a starting point $x_0 \in X$, we choose $\phi(\cdot)=\frac{1}{2}\norm{\cdot-x_0}_p^2$ and $\psi(\cdot)= \frac{1}{2}\norm{\cdot}^2_{p}$, which are $(p-1)$-strongly convex with respect to $\norm{\cdot}_{p}$ \cite[Proposition~3.5]{juditsky2008large}. It is straightforward to show that $\phi^*(\cdot)=\frac{1}{2}\norm{\cdot}_q^2 + \inner{\cdot}{x_0}$ and $\psi^*(\cdot)= \frac{1}{2}\norm{\cdot}^2_{q}$ and that $\phi^*$ and $\psi^*$ are differentiable on $\mathbb{R}^n$ with
{\small
\[
\nabla \phi^*(u)=\norm{u}_q^{2-q}\left(\lvert u_1\rvert^{q-1},\dots, \lvert u_n \rvert^{q-1} \right) + x_0 ,\quad
\nabla \psi^*(u)=\norm{u}_q^{2-q}\left(\lvert u_1\rvert^{q-1},\dots, \lvert u_n \rvert^{q-1} \right).
\]
}

\begin{corollary} \label{cor::grad_norm_p<2}
Let $p\in (1,2]$ and $f$ be an $L$-smooth convex function with respect to $\norm{\cdot}_{p}$. Let $x_0 \in X$ be a given starting point. The concatenation of \ref{eqn::AMD} and \ref{eqn::AMD-G} with $\sigma=\tau=p-1$, distance-generating functions $\phi(\cdot)=\frac{1}{2}\norm{\cdot-x_0}^2_p$ and $\psi(\cdot)=\frac{1}{2}\norm{\cdot}^2_{p}$, starting point $y_0=\nabla \phi(x_0)=0$, and $\theta_{i}^2 - \theta_{i} = \theta_{i-1}^2$ for $i=0,\dots, N-1$ exhibits the rate 
\begin{align*}
    \norm{\nabla f(x_{2N})}_{q} \leq \frac{{L}\norm{x_0- x_\star}_p}{(p-1) \theta_N^2}=\cO\left(\frac{L\norm{x_0- x_\star}_p}{ N^2} \right).
\end{align*}
\end{corollary} 
\begin{proof}
Direct consequence of Corollary~\ref{cor:amd+amdg}.
\qed\end{proof}

\paragraph{Lower bound.}
Our upper bound of Corollary~\ref{cor::grad_norm_p<2} matches the $\Omega\left(\frac{L\norm{x_0- x_\star}_p}{ N^2} \right)$ lower bound of \cite[Corollary~2]{diakonikolas2023complementary} up to a constant factor and is therefore optimal. Our upper bound improves upon the prior state-of-the-art result for this setup \cite[Theorem~3]{diakonikolas2023complementary}, which we now know is loose by a logarithmic factor. For $p\in (2,\infty)$, however, $\phi(\cdot)=\frac{1}{2}\norm{\cdot}^2_p$ is not strongly convex with respect to $\norm{\cdot}_p$ with a dimension-independent strong convexity parameter, as discussed in \cite[Section~1]{diakonikolas2023complementary}, so the concatenation of \ref{eqn::AMD} and \ref{eqn::AMD-G} cannot yield a dimension-independent upper bound. We leave the investigation of the $p\in (2,\infty)$ case to future work.

\subsection{Entropy-regularized optimal transport}
Consider the discrete optimal transport problem
\begin{align*}
\begin{array}{ll}
\underset{X \in {\RR_{+}^{m \times n}}}{\mbox{minimize}}&\quad \inner{C}{X}\\
\mbox{subject to}&\quad X\mathbf{1}=\mu,\,\, X^\intercal \mathbf{1}= \nu,
\end{array}
\end{align*}
where $C \in \RR^{m \times n}$ is the transport costs with nonnegative entries, 
% {\color{blue} $C\in \RR_{+}^{m \times n}\colon = \{M\in \RR^{m \times n} \,|\, M_{ij} >0,\forall \,i,j \in [1,m] \times [1,n]  \}$ is the transport costs}
$\mu\in \RR^m$ and $\nu\in \RR^n$ represent probability mass functions with full support, i.e.,
\[
\mu_1+\dots+\mu_m=1,
\quad
\mu_i>0\text{ for }i=1,\dots,m
\]
and likewise for $\nu$,
$\inner{C}{X}=\sum_{i=1}^m\sum_{j=1}^nC_{ij}X_{ij}$, and $\mathbf{1}$ denotes the vectors in $\mathbb{R}^m$ or $\mathbb{R}^n$ with all components $1$. Write $X_\star$ to denote a solution to this (non-regularized) problem.

The entropy-regularized discrete optimal transport problem with parameter $r>0$ \cite{marco2013sinkhorn,fang1992unconstraineed,lin2022efficiency} is 
\[
\begin{array}{ll}
\underset{X \in {\RR_{+}^{m \times n}}}{\mbox{minimize}}&\quad 
    \inner{C}{X} + r\inner{X}{\log(X)-U}\\
\mbox{subject to}&\quad X\mathbf{1}=\mu,\,\, X^\intercal \mathbf{1}= \nu,
\end{array}
\]
where $\log(X)$ denotes the element-wise logarithm, we use the convention $0\cdot \log(0)=0$, and $U\in \RR^{m \times n}$ is the matrix with all entries $1$. 
The corresponding regularized dual problem is
\[
\begin{array}{ll}
\underset{u \in \RR^m,\, v \in \RR^n}{\mbox{minimize}}&\quad 
     h(u,v) =r\log \Big( \sum_{i,j} \exp\big(\frac{u_i+v_j-c_{ij}}{r} \big)\Big)-\inner{\mu}{u} - \inner{\nu}{v}.
\end{array}
\]
The $h(u,v)$ function of the regularized dual problem is $(1/r)$-smooth with respect to $\|\cdot\|_\infty$.

If we solve the regularized dual problem approximately in the sense of making $\norm{\nabla h(u,v)}_1$ small, then we can obtain an approximate solution to the (non-regularized) primal problem using the following lemma.
\begin{lemma}\cite[Theorem~1, Lemma~7]{altschuler2018nearlinear}\label{lemma:OT_1}
For $(u,v) \in \RR^{m} \times \RR^n$, define the $m\times n$ matrix $B(u,v)_{ij} = \exp\left(\frac{u_i+v_j-c_{ij}}{r}\right)$ for $(i,j)\in [1,m] \times [1,n]$. Define $X(u,v)=B(u,v)/\inner{U}{B(u,v)}$. Then,
\begin{align*}
    \inner{C}{X(u,v)-X_\star} \leq r\log(mn) + 2\norm{C}_\infty \norm{\nabla h(u,v)}_1.
\end{align*}
Moreover, there exists an algorithm with running time $\mathcal{O}\left(mn \right)$ that takes an input $X(u,v)$ and outputs $\hat{X}(u,v)$, which satisfies the constraints
$\hat{X}(u,v)\mathbf{1}=\mu$ and $\hat{X}(u,v)^\intercal \mathbf{1}= \nu$ and satisfies $\|\hat{X}(u,v)-X(u,v)\|_1 \leq 2\norm{\nabla h(u,v)}_1$.
\end{lemma}

Specifically, set $r=\frac{\varepsilon}{2\log(mn)}$ and find $(u_0,v_0)$ such that
$\norm{\nabla h(u_0,v_0)}_1 \leq  \frac{\varepsilon}{8\norm{C}_\infty}$.
By Lemma~\ref{lemma:OT_1}, this is sufficient, since
\begin{align*}
    \langle C,\hat{X}(u_0,v_0)-X_\star\rangle &\leq r\log(mn) + 2\norm{C}_\infty \norm{\nabla h(u_0,v_0)}_1  + \inner{C}{X(u_0,v_0)-X_\star} \\
    &\leq r\log(mn) + 4\norm{C}_\infty \norm{\nabla h(u_0,v_0)}_1 \\
    & \leq \varepsilon.
\end{align*}

Take $\phi(x) = \psi(x)=\frac{1}{2}\norm{x}^2_2$, which is $1$-strongly convex with respect to $\|\cdot\|_\infty$. With a starting point $(0,0)$, implement the concatenation of \ref{eqn::AMD} and \ref{eqn::AMD-G} and denote the output as $(u_{2N},v_{2N})$. By Corollary~\ref{cor:amd+amdg},
\begin{align*}
    \norm{\nabla h(u_{2N},v_{2N})}_2^2  = \mathcal{O}\Big(\frac{\norm{(u_\star,v_\star)}_2 ^2}{r^2 N^4}\Big),
\end{align*}
where $(u_\star,v_\star)\in \argmin_{u,v}h(u,v)$.\ 
\begin{lemma}\cite[Lemma~3.2]{lin2022efficiency} \label{lemma:OT_2}
 If $\mu$ and $\nu$ have full support, then there exists an optimal point $(u_\star,v_\star)$ of $h(u,v)$ such that $\norm{(u_\star,v_\star)}_\infty = \mathcal{O}\left( \norm{C}_\infty\right)$.
\end{lemma}
Using Lemma~\ref{lemma:OT_2} and $\norm{(u,v)}_1 \leq (m+n)^{1/2}\norm{(u,v)}_2\leq (m+n)\norm{(u,v)}_\infty$ for $(u,v) \in \RR^{m} \times \RR^n$, we get 
\begin{align*}
    &\norm{\nabla h(u_{2N},v_{2N})}_1 \leq (n+m)^{1/2}\norm{\nabla h(u_{2N},v_{2N})}_2 
    \\& \qquad = \mathcal{O}\left(\frac{(m+n)^{1/2}\norm{(u_\star,v_\star)}_2}{r N^2}\right)  = \mathcal{O} \left(\frac{(m+n)\log(mn)\norm{C}_\infty}{\varepsilon N^2} \right).
\end{align*}
To ensure $\norm{\nabla h(u_{2N},v_{2N})}_1 \leq  \mathcal{O}\big(\frac{\varepsilon}{\norm{C}_\infty}\big)$, set
\[
\frac{(m+n)\log(mn)\norm{C}_\infty}{\varepsilon N^2} \le 
\frac{\varepsilon}{\norm{C}_\infty}
\]
and solve for $N$ to conclude that $\mathcal{O}\left((m+n)^{1/2}\log(mn)\norm{C}_\infty/\varepsilon \right)$ iterations are required. Finally, since each step of \ref{eqn::AMD-G} and \ref{eqn::AMD} require $\mathcal{O}(mn)$ arithmetic operations, the total arithmetic complexity needed to obtain an $\varepsilon$-approximate solution for the (non-regularized) discrete optimal transport problem is
\[
\mathcal{O}\Big(\frac{(m+n)^{5/2}\log(mn)\norm{C}_\infty}{\varepsilon}\Big).
\]

This rate improves upon the rate of \cite[Section~5.6]{diakonikolas2023complementary} by logarithmic factors and matches the state-of-the-art ${\mathcal{O}}((m+n)^{5/2}\log(mn)\norm{C}_\infty/\varepsilon)$ rate of \cite{antonin2022accelerated}, among non-parallel methods.

\section{Comparison with H-duality}
\label{ss:h-duality-relation}

In this section, we discuss the relationship between our proposed mirror duality and the H-duality of prior work \cite{kim2023timereversed}. First, we show that the \emph{mirror dual} operation reduces to the H-dual operation in the Euclidean case. 

\begin{proposition} \label{prop:H-dual}
Let $X$ be an Euclidean space equipped with the Euclidean norm (or $X$ is a Hilbert space). Let $\phi=\psi=(1/2)\|\cdot\|_2^2$ and $\sigma=1$. For a CFOM \eqref{eqn::CFSFOM}, assume $1 - \sum_{i=0}^{k}\sum_{j=0}^{i+1} b_{i+1,j} =1$ for $k=0,\dots,N-1$, which is a sufficient condition for $x_k\in \mathrm{conv}\{\nabla \phi^{*}(y_0),\dots,\nabla \phi^{*}(y_k)\}$ for $k=0,\dots , N$. 
\footnote{
If the CFOM \eqref{eqn::CFSFOM} does not satisfy  $1 - \sum_{i=0}^{k}\sum_{j=0}^{i+1} b_{i+1,j} =1$ for $k=0,\dots,N-1$, then it reduces to an $N$-step FSFOM of the form
\[
x_{k+1} = \gamma_kx_k - \frac{1}{L}\sum_{i=0}^{k} h_{k+1,i}\nabla f(x_i), \quad k=0,1,\dots, N-1,
\]
where $\gamma_k\neq 1 $ for some $k$.
} Then the CFOM \eqref{eqn::CFSFOM} reduces to an $N$-step FSFOM of the form
\[
x_{k+1} = x_k - \frac{1}{L}\sum_{i=0}^{k} h_{k+1,i}\nabla f(x_i), \quad k=0,1,\dots, N-1,
\]
and the mirror dual CFOM \eqref{eqn::dual_alg_def} reduces to an $N$-step FSFOM of the form
\[
q_{k+1} = q_k- \frac{1}{L}\sum_{i=0}^{k} h_{N-i,N-1-k}\nabla f(q_i), \quad k=0,1,\dots, N-1,
\]
so the FSFOMs are H-duals of each other.
\end{proposition}

\begin{proof}
For any matrix $M \in \RR^{N\times N}$, we can describe the relationship between $M$ and $M^\AT$ via introducing  $J \in \RR^{N\times N}$ as $J_{i,j}=1$ if $i+j=N+1$ and $J_{i,j}=0$ otherwise, so that
\begin{align} \label{eqn::appendix_AT_property}
    M^\AT = JM^{\intercal}J.
\end{align}
Moreover, for $N\times N$ matrices $X$ and $Y$,
\begin{align} \label{eqn::appendix_AT_property_2}
    \left(XY\right)^\AT = J\left( XY\right)^\intercal J = JY^\intercal X^\intercal J = JY^\intercal J \cdot  J X^\intercal J = Y^\AT X^\AT.
\end{align}
This derivation employs \eqref{eqn::appendix_AT_property} and $J^2 = I$.

Define matrices $A,B\in \RR^{N \times N}$ as $A_{k+1,i+1}=a_{k+1,i+1}$, $B_{k+1,i+1}=b_{k+1,i}$ and otherwise $0$. Since $\nabla \psi^*(y_i)=y_i$,
\begin{align*}
    y_{k+1} = & y_k - \frac{1}{L}\sum_{i=0}^{k} a_{k+1,i}\nabla f(x_i), \\
 x_{k+1} = & x_k - \sum_{i=0}^{k+1} b_{k+1,i}y_i \\= & x_k +\frac{1}{L} \sum_{i=0}^{k+1} b_{k+1,i}\left(-y_0 + \sum_{j=0}^{i-1}\sum_{l=0}^{j} a_{j+1,l}\nabla f(x_l) \right)
 \\ = & x_k +\frac{1}{L} \sum_{i=0}^{k+1} b_{k+1,i}\left(-x_0 + \sum_{j=0}^{i-1}\sum_{l=0}^{j} a_{j+1,l}\nabla f(x_l) \right) \\ =& x_k + \frac{1}{L} \sum_{i=0}^{k+1} b_{k+1,i}\left(\sum_{j=0}^{i-1}\sum_{l=0}^{j} a_{j+1,l}\nabla f(x_l)\right) - \frac{1}{L}\sum_{i=0}^{k+1} b_{k+1,i}x_0 \\
 =& x_k + \frac{1}{L} \sum_{i=0}^{k+1} b_{k+1,i}\left(\sum_{j=0}^{i-1}\sum_{l=0}^{j} a_{j+1,l}\nabla f(x_l)\right).
\end{align*} 
for $k=0,\dots,N-1$. $\sum_{i=0}^{k+1} b_{k+1,i}x_0$ vanishes to $0$, due to \eqref{eqn::appendix_convex_sum_pf_1}. 
The coefficient of $\nabla f(x_l)$ is $h_{k+1,l} =  \sum_{i=0}^{k+1} \left( b_{k+1,i}\sum_{j=0}^{i-1} a_{j+1,i}\right)$. To simplify the above equation, define $H,C \in \RR^{N \times N}$ as $\left[H_{k,i}\colon =h_{k+1,i}\,\,\text{and otherwise $0$}, \quad C_{k,i}\colon =1_{i \leq k-1}\right]$. Then $ H_{k,i} = \sum_{i=0}^{k}\left( B_{k,i}\sum_{j=0}^{i-1}A_{j,i}\right) = \sum_{0 \leq i,j\leq N} B_{k,i}C_{i,j}A_{j,i}$ thus $H=BCA$. For the mirror dual, we obtain
\begin{align*}
    q_{k+1} = q_k -\frac{1}{L}\sum_{i=0}^{k} a_{N-i,N-1-k}r_i, \quad r_{k+1}= r_k -\sum_{i=0}^{k+1} b_{N-i,N-1-l}\nabla f(q_i).
\end{align*}
In the same way, and further using $b_{N,N}=-1$ gives
$H' = A^\AT C  B^\AT$ where $H'_{k,i}=h'_{k+1,i}$. Finally $  \left(H'\right)^\AT = \left(B^\AT\right)^\AT C^\AT \left(A^\AT\right)^\AT = BCA = H$.
\qed\end{proof}

\paragraph{Mirror duality and H-duality.}
Consider when the distance-generating function is the Euclidean norm, \( \phi = \psi = \frac{1}{2} \| \cdot \|^2 \). Proposition~\ref{prop:H-dual} indicates that the mirror dual operation generalizes the H-dual operation. Furthermore, since the Bregman divergence with respect to the Euclidean norm
equals the squared Euclidean distance, our mirror duality theorem generalizes the H-duality theorem \cite[Theorem 1]{kim2023timereversed}.
 However, strictly speaking, the two theorems technically differ.  The energy functions $\{\cU_k\}_{k=0}^{N}$ and $\{\cV_k\}_{k=0}^N$ used in our mirror duality theorem and the H-duality theorem of \cite{kim2023timereversed} are similar, but not the same, primarily due to a technical constraint arising from the inability to expand the norm into an inner product in general Banach spaces.
\section{Conclusion}
\label{s:conclusion}
In this work, we presented mirror duality, a one-to-one correspondence between mirror-descent-type methods reducing function value and those reducing gradient magnitude, thereby advancing our understanding of the task of reducing gradient magnitude. Furthermore, we used mirror duality to obtain \ref{eqn::AMD-G}, the first method that reduces gradient magnitude at an optimal accelerated rate in the mirror-descent setup. 

The results of this work inspire some interesting directions for future work. One direction would be to consider the relaxed assumption of a weakly smooth objective function $f$ and a uniformly convex distance-generating function $\phi$. For the primal setup of reducing function value, \cite{alexandre2018optimal} provides an accelerated mirror-descent-type method. For a related composite minimization setup, \cite{diakonikolas2023complementary} provides an accelerated first-order method. Extending mirror duality and \ref{eqn::AMD-G} to the weakly smooth and uniformly convex setup will also have some interesting applications. 
Another potential direction is to extend \ref{eqn::dual_MD} and \ref{eqn::AMD-G} to the constrained setup in which the notion of mirror descent and dual averaging differ. Yet another direction of future work is to consider the stochastic setup. For the setup of using stochastic gradients of convex functions, \cite{zhu2018how} provides a rate for reducing the Euclidean norm of gradients. Investigating whether similar rates are attainable in the mirror descent setup or whether there is an analog of mirror duality in the stochastic setup would be interesting.

\section*{Acknowledgement}
We thank TaeHo Yoon for reviewing the manuscript and providing valuable feedback. We thank Ya-Ping Hsieh, Anna Korba, Emanuel Laude, and Francesco Orabona for helpful discussions that improved this work. JD acknowledges funding from the NSF grant 2007757.

\bibliographystyle{abbrv}
\bibliography{main}

\newpage
\appendix
\begin{appendix}
\section{Omitted proofs}
\label{Appendix:Banach}
\paragraph{The proof of Lemma~\ref{lemma::amd_and_dual_amd}.} Using the expression of $x_k$ that we used in Fact~\ref{fact:amd} provides us
\begin{align*}
    x_k - x_{k+1} &= \left( \frac{\theta_{k}^2 -\theta_{k-2}^2}{\theta_{k}^2}\nabla \phi^{*}(y_{k})  + \sum_{s = 1}^{k-1}  \frac{\theta_{s-1}^2 -\theta_{s-2}^2}{\theta_{k}^2}\nabla \phi^*(y_{s})\right)
    \\
    &\qquad \qquad \qquad - \left(\frac{\theta_{k+1}^2 -\theta_{k-1}^2}{\theta_{k+1}^2}\nabla \phi^*(y_{k+1})  + \sum_{s = 1}^k  \frac{\theta_{s-1}^2 -\theta_{s-2}^2}{\theta_{k+1}^2}\nabla \phi^*(y_{s}) \right)
    \\
    &= \sum_{s = 1}^{k-1}  \left( \frac{\theta_{s-1}^2 -\theta_{s-2}^2}{\theta_{k}^2} -  \frac{\theta_{s-1}^2 -\theta_{s-2}^2}{\theta_{k+1}^2} \right) \nabla \phi^*(y_{s})\\
    &\quad + \left(\frac{\theta_{k}^2 -\theta_{k-2}^2}{\theta_{k}^2}- \frac{\theta_{k-1}^2 -\theta_{k-2}^2}{\theta_{k+1}^2}\right)\nabla \phi^*(y_{k}) - \frac{\theta_{k+1}^2 -\theta_{k-1}^2}{\theta_{k+1}^2}\nabla \phi^*(y_{k+1}).
\end{align*}
Thus $b_{k+1, 0}^{(\text{AMD})} = 0$, $b_{k+1, s}^{(\text{AMD})} =  \frac{\theta_{s-1}^2 -\theta_{s-2}^2}{\theta_{k}^2} -  \frac{\theta_{s-1}^2 -\theta_{s-2}^2}{\theta_{k+1}^2}$ for $1 \leq s \leq k-1$, $b_{k+1, k}^{(\text{AMD})} = \frac{\theta_{k}^2 -\theta_{k-2}^2}{\theta_{k}^2}- \frac{\theta_{k-1}^2 -\theta_{k-2}^2}{\theta_{k+1}^2}$, and $b_{k+1, k+1}^{(\text{AMD})} = - \frac{\theta_{k+1}^2 -\theta_{k-1}^2}{\theta_{k+1}^2} $ holds. Moreover, we have
\begin{align*}
    a_{k+1, s}^{(\text{AMD})} = 0 \text{ for } s = 0, 1, \dots, k-1, \qquad a_{k+1, k}^{(\text{AMD})} = \frac{\sigma}{L}(\theta_k^2 - \theta_{k-1}^2).
\end{align*} 
Now, we calculate the step sizes of \ref{eqn::AMD-G}. For $k = 0,1, \dots, N-1$, 
\begin{align}
    r_{k+1} &= \sum_{s = 0}^{k} \left((\theta_{N-s-1}^2 - \theta_{N-s-2}^2)\left(g_{s+1}-g_s \right)+ (\theta_{N-s-2}^2 - \theta_{N-s-3}^2)g_{s+1}\right) + r_0 \nonumber
    \\
    &= (\theta_{N-k-2}^2 - \theta_{N-s-3}^2)g_{k+1}+  \sum_{s=1}^{k+1}(\theta_{N-s}^2 - \theta_{N-s-1}^2) g_s  - (\theta_{N-1}^2 - \theta_{N-2}^2) g_0 + r_0 \label{eqn:rNexpand}.
\end{align}
In addition, for $s = 0,1, \dots, N$, the update of \ref{eqn::AMD-G} provides  
\begin{align}
    g_s &= g_0 + \sum_{l = 0}^{s-1}\frac{1}{\theta_{N-l-1}^2}\left(\nabla f(q_{l+1}) -\nabla f(q_l) \right) \nonumber
    \\
    &= g_0 + \sum_{l = 1}^{s-1}\left(\frac{1}{\theta_{N-l}^2} - \frac{1}{\theta_{N-l-1}^2}\right) \nabla f(q_{l}) + \frac{1}{\theta_{N-s}^2}\nabla f(q_{s}) - \frac{1}{\theta_{N-1}^2}\nabla f(q_0) \label{eqn:gkexpand}.
\end{align}
Now, using \eqref{eqn:rNexpand}, we rewrite $r_k - r_{k+1}$ with the summation of $\nabla f(q_l)$ for $k = 0,1, \dots, N-1$ as follows:
\begin{align*}
    b_{k+1, 0}^{(\text{AMD-G})} &= (\theta_{N-k-2}^2- \theta_{N-k-3}^2)\frac{1}{\theta_{N-1}^2}  - (\theta_{N-k+1}^2 - \theta_{N-k}^2)\frac{1}{\theta_N^2}     
    \\
b_{k+1, l}^{(\text{AMD-G})} &= - (\theta_{N-k-2}^2 -\theta_{N-k-3}^2) \left(\frac{1}{\theta_{N-l}^2} - \frac{1}{\theta_{N-l-1}^2}\right)
\\
    b_{k+1, k}^{(\text{AMD-G})} &= (\theta_{N-k}^2 - \theta_{N-k-2}^2)  \frac{1}{\theta_{N-k}^2}  -  (\theta_{N-k-1}^2 - \theta_{N-k-3}^2)\left(\frac{1}{\theta_{N-k}^2} - \frac{1}{\theta_{N-k - 1}^2} \right) 
    \\
    &\qquad - (\theta_{N-k}^2 - \theta_{N-k-1}^2) \frac{1}{\theta_{N-k}^2}
    \\
    &= (\theta_{N-k-3}^2 - \theta_{N-k-2}^2) \frac{1}{\theta_{N-k}^2}  + (\theta_{N-k-1}^2 - \theta_{N-k-3}^2) \frac{1}{\theta_{N-k - 1}^2}  
    \\
   b_{k+1, k+1}^{(\text{AMD-G})} &= -(\theta_{N-k-1}^2 - \theta_{N-k-3}^2)  \frac{1}{\theta_{N-k-1}^2}.    
\end{align*}
Also, by the update rule of \ref{eqn::AMD-G}, we have 
\begin{align*}
q_{k+1} = q_k - \frac{\sigma}{L}(\theta_{N-k-1}^2 - \theta_{N-k-2}^2) \nabla \psi^* (r_k),
\end{align*}
so we can specify $a_{k+1, i}^{(\text{AMD-G})}$ as follows:
\begin{align*}
    a_{k+1, s}^{(\text{AMD-G})} = 0 \text{ for } s = 0, 1, \dots, k-1, \qquad a_{k+1, k}^{(\text{AMD-G})} =  \frac{\sigma}{L}(\theta_{N-k-1}^2 - \theta_{N-k-2}^2).
\end{align*} 
Therefore, we conclude that the mirror dual of \ref{eqn::AMD} is \ref{eqn::AMD-G}.   
\end{appendix}
\end{document}